\documentclass[10pt,twoside]{article}
\usepackage{amsmath, amsfonts, amsthm, amssymb, amscd, a4wide}
\usepackage{ color, graphicx}
\usepackage{mathrsfs, dsfont,latexsym}
\usepackage{paralist}
\usepackage[mathcal]{euscript} 
\usepackage{dsfont, pxfonts, verbatim} 
\usepackage{bbm}
\numberwithin{equation}{section}  
\usepackage{perpage} 
\MakePerPage{footnote}
\oddsidemargin 	.in
\evensidemargin .in
\topmargin    - 0.75in
\numberwithin{equation}{section}
        \newtheorem{theorem}{Theorem}[section]
        \newtheorem{proposition}[theorem]{Proposition}
        \newtheorem{lemma}[theorem]{Lemma}
        \newtheorem{corollary}[theorem]{Corollary} 
        \newtheorem{definition}[theorem]{Definition} 

\let\oldmarginpar\marginpar
\renewcommand\marginpar[1]{\-\oldmarginpar[\raggedleft\footnotesize #1]
{\raggedright\footnotesize #1}}
\newcommand \etat {\widetilde \eta}

\newcommand \EDscr {\mathscr{E_\star D}} 
\newcommand \sign {\text{sgn}} 
\newcommand \Mscr {\mathscr M}
\newcommand \Escr {\mathscr E}
\newcommand \Escrt {\widetilde{\mathscr E}}
\newcommand \NI {{\bf I}}
\newcommand \NII {{\bf II}}
\newcommand \NIII {{\bf III}}
\newcommand \NIV {{\bf IV}}
\newcommand \NV {{\bf V}}
\newcommand \NVI {{\bf VI}}
\newcommand \rhob {\overline \rho}
\newcommand\DD {\mathbb D}

\newcommand \Dcalh {\widehat{\mathcal D}}
\newcommand \Et {\widetilde E} 
\newcommand \Ft {\widetilde F} 
\newcommand \Gt {\widetilde G} 
\newcommand \Dt {\widetilde D} 
\newcommand \ut {\hskip.02cm \widetilde u \hskip.02cm} 
\newcommand \kappat {\hskip.02cm \widetilde \kappa \hskip.02cm} 
\newcommand \Tbb {{\mathbb T}}

\newcommand \be {\begin{equation}}
\newcommand \bel {\be\label}
\newcommand \ee {\end{equation}}
\newcommand \la \langle
\newcommand \ra \rangle

\newcommand	\RR 		{\mathbb R}  
\newcommand	\TT 		{\mathbb T}  

\newcommand \del {{\partial}}
\newcommand \eps \epsilon
\newcommand \lam {\lambda}

\newcommand \loc {\text{loc}} 

\begin{document}

\title{The finite energy method for compressible fluids.
\\
The Navier-Stokes-Korteweg model}  
\author{Pierre Germain\footnote{Courant Institute of Mathematical Sciences, New York University, 
215 Mercer Street, NY10010 New York, USA.
Email: {\sl pgermain@cims.nyu.edu.}} 
\hskip.13cm and 
Philippe G. LeFloch\footnote{Laboratoire Jacques-Louis Lions \& Centre National de la Recherche Scientifique, 
Universit\'e Pierre et Marie Curie (Paris 6), 4 Place Jussieu, 75252 Paris, France. 
Email : {\sl contact@philippelefloch.org.}
\,
\newline
2000\textit{\ AMS Subject Class.} Primary: 35L65. Secondary: 76L05. 
}
}
\date{December 2012}

\maketitle

\begin{abstract}
This is the first of a series of papers devoted to the initial value problem for the Euler system of 
compressible fluids and augmented versions containing higher-order terms. We encompass solutions that have finite total energy and 
enjoy a certain symmetry (for instance, plane symmetry); these solutions may have unbounded amplitude 
and contain cavitation regions in which the mass density vanishes. 
In the present paper, we are interested in dispersive shock waves and 
analyze the zero viscosity--capillarity limit associated with the Navier-Stokes-Korteweg system.
Specifically, we establish the existence of finite energy solutions as well as their convergence toward entropy solutions to the Euler system. 
We encompass a broad class of nonlinear Navier-Stokes-Korteweg constitutive laws,
which is determined by two main conditions relating the viscosity and capillarity coefficients, that is,  on one hand 
the {\sl strong coercivity condition} (as we call it) which provides a favorable sign for the integrated dissipation associated with an effective energy, 
and on the other hand the {\sl tame capillarity condition} (as we call it), which restricts pointwise the strength of the capillarity relatively to the viscosity. 
Rather mild conditions on the growth of the constitutive functions are aso imposed, which are required in order 
to define finite energy weak solutions  to the  Navier-Stokes-Korteweg system, even in the presence of cavitation. 
Our method of proof relies on fine algebraic properties of the Euler system and combines together energy and effective energy estimates, dissipation and effective dissipation estimates, a nonlinear Sobolev inequality, high--integrability properties for the mass density and for the velocity, and compactness properties based on entropies. 
\end{abstract}
\maketitle 

\tableofcontents 


\section{Introduction and outline}
\label{sec:1}

\subsection{Euler system of compressible fluid flows}

In this paper and the companions \cite{GL2,GL3}, we establish existence, compactness, and convergence results for (augmented versions of) the Euler system of compressible fluid dynamics, when the fluid flow under consideration has plane, cylindrical, or spherical symmetry or the fluid evolves within a nozzle with variable cross-section. The proposed Finite Energy Method, as we call it, encompasses {\sl real fluids} (as the standard class of polytropic fluids is too restrictive in applications) as well as broad classes of {\sl augmented Euler systems,} which incorporate physically-relevant small-scale terms. The proposed method allows us to validate the singular limit problem associated with these models ---in presence of cavitation and shock waves---
and, specifically, to establish the convergence with finite energy solutions to the augmented models toward finite energy solutions to the Euler system. It originates from pioneering works by DiPerna \cite{DiPerna1,DiPerna2} on bounded solutions and by LeFloch and Westdickenberg \cite{LW} on finite energy solutions.  

In the present paper, we rigorously validate the {\sl vanishing viscosity--capillarity method} under mild and physically realistic assumptions, as is now presented. Roughly speaking, we determine conditions under which 
`dispersive shock waves' converge to shock wave solutions to the Euler system when the capillarity and viscosity tend to zero.  Recall that the {\bf Euler system} for isentropic compressible fluid flows (in plane symmetry, for simplicity) reads
\be
\label{Euler0} 
\aligned
\rho_t + (\rho u)_x & = 0, 
\\
(\rho u)_t + (\rho u^2 + p(\rho))_x & = 0,
\endaligned
\ee
in which $t\geq 0$ represents the time variable while the spatial variable $x$ takes its values in either the torus $\Tbb$ or the real line $\RR$. 
The fluid is characterized by its  mass density $\rho=\rho(t,x) \geq 0$ and velocity $u=u(t,x) \in \RR$. 
The two equations \eqref{Euler0} form a nonlinear hyperbolic system of conservation laws, provided the pressure function $p=p(\rho)$ satisfies the monotonicity condition 
\be
\label{hyperc}
p'(\rho) >0 \qquad (\rho>0). 
\ee
Importantly, strict hyperbolicity fails at the vacuum when $\lim_{\rho \to 0} p'(\rho) = 0$. 
The Euler system is also genuinely nonlinear in the sense of Lax \cite{Lax} (away from the vacuum), provided the pressure satisfies the convexity condition 
\be
\label{GNLc}
p''(\rho) + \rho \, p'(\rho) >0  \qquad (\rho>0). 
\ee
These two standard assumptions are made in the present study. (See Lax \cite{Lax} and Dafermos \cite{Dafermos} for a background on nonlinear hyperbolic systems.) 

In continuum physics, compressible fluid flows are often governed by {\sl general equations of state} $p= p(\rho)$ 
and, in addition, the selection of shock waves in such flows may be driven by possibly nonlinear,  {\sl higher--order modeling terms}. The physical models typically 
contain second-order derivatives of the unknowns 
and take into account the effect of the viscosity, capillarity, and Hall term in the fluid under consideration. It is our objective in the present series of papers to develop the mathematical tools  which are required for investigating the existence and properties of complex fluid flows. In particular, it has been recognized in the past fifteen years that ``nonclassical'' wave structures and dispersive shock waves 
do arise in such flows, which are not observed with polytropic fluids and with the vanishing viscosity method, whereas they do play critical role in many physical applications and, therefore, deserve the attention of applied mathematicians. 
We refer to LeFloch~\cite{LeFloch-book} for a review of some mathematical tools, and 
to Section~\ref{sec-section1.4}, below, for numerous references in a variety of areas:
van der Waals fluid dynamics; 
dispersive shallow water flows; quantum hydrodynamic; Bose-Einstein condensates; 
Boussinesq model; Green-Naghdi models.

As far as plane-symmetric solutions to the Euler system are concerned, 
the vanishing viscosity problem was first studied by DiPerna in the pioneering work \cite{DiPerna1,DiPerna2}.  DiPerna relied on Tartar's compensated compactness method \cite{Tartar0,Tartar1} and a compactness embedding theorem by Murat \cite{Murat1,Murat2}. This method was later extended by Ding, Chen, and Luo \cite{DCL}, Morawetz \cite{Morawetz}, Lions, Perthame, Souganidis \cite{LPS}, Lions, Perthame, and Tadmor \cite{LPT}, Chen and LeFloch \cite{CL1,CL2}, and LeFloch and Shelukhin \cite{LeFlochShelukhin}. In these works, 
all solutions under consideration have arbitrary large but  {\sl bounded} amplitude. 

Several years ago, LeFloch and Westdickenberg \cite{LW} opened the way to constructing solutions 
within the broader class of {\sl solutions with finite energy} and realized that working within such a large class of solutions was necessary in order to overcome certain  limitations in DiPerna's theory. In this class, they established the first existence result of {\sl radially symmetric} fluid flows ---including the singularity at the center of radial coordinates. The present series of papers follows this strategy and, by exhibiting suitable algebraic and differential properties of the Euler equations 
and by covering broad classes of constitutive laws, we develop a general tool in order to handle the complex fluid flows arising in physical applications. 
In this first paper, we focus on the zero viscosity--capillarity limit and on flows with  
plane symmetry, and we refer to the companion papers \cite{GL2,GL3} for other aspects of the proposed method. 

Throughout, we adopt the following standard notation. We denote by $C$ a positive constant and by $C(\Escr_0)$ a constant depending on some known quantity $\Escr_0$; the specific value of this constant may change from one occurence to another. Finally, it will be convenient to write $A \lesssim B$ when the inequality $A \leq C \, B$ holds for some constant $C>0$, while we will use the notation $p(\rho) \simeq \rho^\gamma$ if $p(\rho) = C \, \rho^\gamma$. 
The so-called ``Japanese bracket'' $\langle x \rangle$ stands for $\sqrt{1+x^2}$ (for a real variable $x$).  


\subsection{Navier-Stokes-Korteweg system for viscous-capillary fluid flows}
\label{sectionNSK}

We consider the Cauchy problem associated with the {\bf Navier-Stokes Korteweg} NSK system 
\bel{eq:101} 
\aligned
\rho_t + (\rho u)_x =& \, 0, 
\\
(\rho u)_t + \big( \rho u^2 + p(\rho) \big)_x =& \, 
L[\rho,u]_x + K[\rho]_x,
\\
L[\rho,u] :=& \, \mu(\rho) \, u_x, 
\qquad
\qquad 
\\
K[\rho] :=& \, \rho \kappa(\rho) \, \rho_{xx} + \frac{1}{2} \, \big( \rho \kappa'(\rho) - \kappa(\rho) \big) \, \rho_x^2, 
\endaligned
\ee
with initial conditions  
\bel{eq:101b-weak}
(\rho,\rho u)|_{t=0} = (\rho_0, \rho_0u_0) \quad \text{ on } \DD
\ee
for some prescribed functions $\rho_0 \geq 0$ and $u_0$ defined on $\DD$. 
Here, the nonlinear viscosity $\mu = \mu(\rho)$ and the nonlinear capillarity $\kappa=\kappa(\rho)$ are prescribed and smooth functions in $\rho>0$, which may be singular at the vacuum. 
Throughout and for definiteness, we always assume that the following (mild) assumptions are satisfied:  
\begin{subequations}
\label{hypmukappa}
\begin{align}
& \aligned
& (1) \quad \mu(\rho) > 0 \quad \text{(for all $\rho>0$),} 
\\
& (2) \quad \lim_{\rho \to 0} \mu(\rho) = 0,        && \quad \liminf_{\rho \to +\infty} \mu(\rho) >0,
\endaligned 
\\
& (3) \quad 
\mbox{ either } \, \kappa \equiv 0 \quad 
 \mbox{or } \quad 
\kappa(\rho) > 0 \, \text{ (for all $\rho>0$),} 
\\
& (4) \quad \rho \, p' \lesssim p, \qquad \rho \, | \mu'| \lesssim \mu, 
\qquad \rho \, | \kappa'| \lesssim \kappa. 
\end{align}
\label{s:geodesics}\end{subequations}

The (mass density and velocity) dependent variables $(\rho,u) \in \RR_+ \times \RR$ are unknown functions of $(t,x) \in \RR^+ \times \DD$.
We distinguish between two cases: 
\begin{itemize}
\item $\DD = \TT$ and the NSK system is thus set on the torus or, equivalently, a bounded interval with periodic boundary conditions.
\item $\DD = \RR$ and the NSK system is set on the real line, subject to the following asymptotic condition, 
for a given constant mass density $\rho_\star \geq 0$, 
\bel{eq:cd}
\lim_{|x| \to +\infty} (\rho, \rho u) (t,x) = (\rho_\star,0), \qquad t  \geq 0. 
\ee
\end{itemize}

Solutions to \eqref{eq:101}  may exhibit cavitation phenomena, in the sense that the mass density $\rho$ may vanish. While this phenomena is unavoidable in solutions to the Euler system \eqref{Euler0}, it can be avoided in \eqref{eq:101}, {\sl provided} the viscosity or capillarity are sufficiently ``strong'' near the vacuum.  
While it is straighforward to define finite energy weak solutions when the mass density remains bounded away from zero, 
we will need to introduce (in Section~\ref{sec:3}) a suitable {\sl notion of weak solutions with cavitation} in order to handle more general viscosity and capillarity coefficients. 
Observe that if $\rho_0>0$ we can state the initial condition as 
\bel{eq:101b}
(\rho, u)|_{t=0} = (\rho_0, u_0) \qquad \text{ on } \DD. 
\ee

Consider first the equations \eqref{eq:101} posed on the torus $\TT$. 
The local physical energy associated with the NSK system (cf.~Section~\ref{sec:energy}, below) reads  
\bel{en-cap}
E[\rho,u] := \frac{1}{2} \rho u^2 + \rho e(\rho) + \frac{1}{2} \kappa(\rho) \rho_x^2,
\ee
in which the internal energy $e=e(\rho)$ is defined by $e'(\rho) : = p(\rho) / \rho^2$, while the total energy within $\TT$ is defined as 
$$
\Escr[\rho,u](t) := \int_{\TT} E[\rho,u](t, \cdot) \, dx. 
$$
Interestingly enough, the NSK system is endowed with {\sl another} energy functional, which we refer to here as the {\sl effective energy}.
It is based on the {\bf effective velocity,} defined by 
\bel{en-vel}
\ut := u + \frac{\mu(\rho)}{\rho^2} \rho_x. 
\ee
More precisely, the local effective energy is $\Et[\rho,u] = E[\rho,\ut]$ and the total effective energy is defined to be 
$$
\Escrt[\rho,u] := \Escr [\rho,\ut].
$$
Observe that both $\Escr$ and $\Escrt$ control $\kappa^{1/2} \rho_x$ in $L^2$; and that a combination of $\Escr$ and $\Escrt$ controls $\frac{\mu}{\rho^{3/2}}\rho_x$ in $L^2$.

Consider next the Cauchy problem \eqref{eq:101}--\eqref{eq:101b} posed on the real line $\RR$. The local physical energy
needs to be renormalized in agreement with \eqref{eq:cd} as follows: 
\bel{eq:Estar}
\aligned
E_\star[\rho,u] :=
& \frac{1}{2} \, \rho u^2 + \rho e(\rho) - \rho_\star e(\rho_\star) - \big( e(\rho_\star)+ \rho_\star e'(\rho_\star) \big) \, (\rho - \rho_\star) + \frac{1}{2} \kappa(\rho) \rho_x^2, 
\\
\Escr_\star[\rho,u] :=& \int_{\RR} E_\star[\rho,u]\, dx, 
\endaligned
\ee
while the functionals $\Et_\star$ ans $\Escrt_\star$ are defined similarly. Of course, the expressions given on the torus coincide with the ones on the real line by replacing $\rho_\star$ by $0$. 
From now on, we therefore use the latter notation {\sl in both cases.} 

Our result about the NSK system concerns the Cauchy problem with initial data $(\rho_0, u_0): \DD \to [0, +\infty) \times \RR$
with finite total (physical and effective) energy, 
i.e.
\be
\label{sittelle}
\Escr_\star[\rho_0,u_0] + \Escrt_\star[\rho_0,u_0] < + \infty.
\ee

In the course of writing this paper, we discovered a natural restriction on the capillarity coefficient of the Navier-Stokes-Korteweg system, which we refer to as the {\sl Strong Coercivity} (SC) condition.
It turns out to play a fundamental role in the mathematical theory; see Section~\ref{sec:SCC}, below, and especially \eqref{eq:203} and \eqref{eq:203bis}.
Under the condition (SC) and for all finite energy initial data, we will establish the following {\bf finite total energy-dissipation estimate:}   
\be
\label{cigogne}
(\EDscr)\qquad  
\begin{cases} \quad
\aligned
& \sup_{t \geq 0} \big( \Escr_\star[\rho,u] (t)+ \Escrt_\star[\rho,u](t) \big) 
 + \iint_{(0,\infty)\times \DD} \mu(\rho) u_x^2 \, dtdx
\\
& + \iint_{(0,\infty)\times \DD} 
\Bigg(
\frac{\mu(\rho) p'(\rho)}{\rho^2} (\rho_x)^2 + \frac{\mu(\rho) \kappa(\rho)}{\rho} \rho_{xx}^2 + \frac{\mu(\rho)\kappa(\rho)}{\rho^3} \rho_x^4 \Bigg)
\, dtdx 
\\&
\lesssim \Escr_\star[\rho_0,u_0] + \Escrt_\star[\rho_0,u_0],
\endaligned
\end{cases}
\ee
for which we refer to the following section.

Before stating the theorem, we also introduce the {\bf no-cavitation (NC)} condition 
\be
\label{behave}
(NC) \qquad 
\aligned
& \mbox{For some $s>0$,} \quad 
  \liminf_{\rho \to 0}  \rho^{s-1/2} \mu(\rho) >0
\quad
\text{ or }
\quad 
\liminf_{\rho \to 0} \rho^{2+s} \kappa(\rho) >0. 
\endaligned
\ee
As we show below, the condition (NC) ensures that vacuum does not appear in the fluid if it is not present initially. However, 
the following theorem covers, both, the non-cavitating and the cavitating regimes. 

\begin{theorem}[Weak solutions to the Navier--Stokes--Korteweg system.] 
\label{theo:11} 

Consider the Cauchy problem \eqref{eq:101}--\eqref{eq:101b} associated with the Navier-Stokes Korteweg system, when the spatial domain $\DD$ is either 
the torus or the real line and, in the latter case, the condition \eqref{eq:cd} is assumed for some $\rho_\star \geq 0$.
Assume that the initial data $\rho_0, u_0$ have finite total physical energy and effective 
energy (see \eqref{sittelle}), and that the viscosity and capillarity coefficients satisfy the strong coercivity condition (SC). Finally, consider one of the following four set-ups: 
\begin{itemize}

\item[(i)] The no-cavitation condition (NC) is satisfied and $\rho_0(x)>0$ for all $x \in \DD$.

\item[(ii)] $\DD = \TT$; the capillarity coefficient $\kappa \equiv 0$ vanishes identically; 
the integral $\int_\DD \rho_0 |u_0|^{2+s} \, dx$ is finite; for some $s>0$,  
the inequality $p(\rho)^2 \lesssim \rho^{s/2} \mu(\rho)$ holds for small $\rho>0$. 

\item[(iii)] $\DD = \RR$ with $\rho_\star=0$; $\kappa \equiv 0$; for some $s>0$, the inequalities $\rho \lesssim \mu(\rho) \lesssim \rho^s$ hold 
for small $\rho>0$; 
the total mass of the fluid $\Mscr[\rho_0] := \int_{\RR} \rho_0 \, dx$ is finite.

\item[(iv)] $\DD = \RR$ with $\rho_\star \geq 0$; the viscosity-capillarity pair $(1,\kappa)$ satisfies (SC); 
the inequalities $\mu(\rho) \lesssim \rho^{2/3}$ and $\kappa(\rho) \lesssim \frac{\mu^2(\rho)}{\rho^3}$ hold for small $\rho>0$. 

\end{itemize}
Then, the Cauchy problem for the Navier-Stokes Korteweg system  \eqref{eq:101} 
admits a global-in-time weak solution $(\rho,u): [0, +\infty) \times \DD \to [0, +\infty) \times \RR$ which has finite total energy
and dissipation in the sense $(\EDscr)$.
In the non-cavitating case (i), one also has $\rho(t,x)>0$ for all $(t,x)$.  
\end{theorem}

Observe that the right-hand side of the NSK system must be understood in the sense of distributions. 
Theorems~\ref{guillemot1}, \ref{guillemot2}, \ref{guillemot2bis}, and~\ref{guillemot3}, stated and proven below,
correspond to the above four cases, respectively.  

While many well--posedness results are available in the literature for augmented versions of the Euler system such as the NSK system \eqref{eq:101}, 
as well as variants and multi--dimensional generalizations, most papers, however, restrict attention to the purely viscous case ($\kappa=0$) and to non-cavitating solutions. We do not attempt here to review the vast literature on the Navier--Stokes system and only quote works on one-dimensional solutions. 
Assuming $\kappa \equiv 0$, Hoff~\cite{Hoff}  first treated the case of constant viscosity and finite energy solutions. 
More recently, Mellet and Vasseur~\cite{MelletVasseur} covered the viscosity coefficients $\mu(\rho) = \rho^\alpha$ with $\alpha<1/2$ (a condition which prevents cavitation) and initial data with finite energy and (in our terminology) finite effective energy. 
More recently, for the same class of initial data and viscosity coefficients, Jiu and Xin~\cite{JX} treated the interval $\alpha>1/2$. 

The study of the Euler equations with a capillarity term was tackled much more recently. 
Under various conditions on the viscosity and capillarity coefficients, strong solutions were constructed in Danchin and Desjardins~\cite{DanchinDesjardins}, Benzoni-Gavage, Danchin, Descombes~\cite{BDD1,BDD2}, and Hao~\cite{Hao}, 
while the existence of weak solutions was established by Bresch and Desjardins~\cite{BreschDesjardins}, J\"ungel~\cite{Jungel},
as well as Gamba, J\"ungel, and Vasseur~\cite{GambaJungelVasseur}. In contrast, the existence theory proposed here relies on 
the strong coercivity condition (SC), only. 


\subsection{The zero viscosity--capillarity limit}

Our second objective is the convergence with finite energy solutions to  (NSK) in the limit of vanishing viscosity and capillarity. 
We therefore replace $\mu, \kappa$ in the right-hand side of \eqref{eq:101} by 
coefficients $\mu^\eps, \kappa^\eps$ depending on a parameter $\eps$ and approaching zero when $\eps \to 0$.
For definiteness and without genuine loss of generality, we rescale the given
 viscosity and capillarity coefficients $\mu, \kappa$, as follows: 
\bel{eq:scaling}
\mu^\eps(\rho) = \eps \mu(\rho), 
\qquad
\qquad \kappa^\eps (\rho) = \delta(\eps) \kappa(\rho),
\ee
where $\delta=\delta(\eps)$ tends to zero with $\eps$.  Then, we consider the finite energy weak solutions $(\rho^\eps, u^\eps)$ associated with the NSK system$_\eps$: 
\bel{eq:102} 
\aligned
\rho^\eps_t + (\rho^\eps u^\eps)_x =& \, 0, 
\\
(\rho^\eps u^\eps)_t + (\rho^\eps (u^\eps)^2 + p(\rho^\eps) \big)_x =& \, L^\eps[\rho^\eps,u^\eps]_x + K^\eps[\rho^\eps]_x,
\\
L^\eps[\rho^\eps,u^\eps] :=& \, \mu^\eps(\rho^\eps) \, u^\eps_x, 
 \quad 
\\
K^\eps[\rho^\eps] :=& \rho^\eps \kappa^\eps(\rho^\eps) \, \rho^\eps_{xx} + \frac{1}{2} \, \big( \rho^\eps {\kappa^\eps}'(\rho^\eps) - \kappa^\eps(\rho^\eps) \big) (\rho_x^\eps)^2, 
\endaligned
\ee
and a family of initial data $\rho^\eps_0, u^\eps_0$: 
\bel{eq:102b}
(\rho^\eps,\rho^\eps u^\eps)|_{t=0} = (\rho^\eps_0, \rho^\eps_0 u^\eps_0).
\ee
We thus consider the singular limit problem $\eps \to 0$ and we are going to establish that finite energy weak solutions to the Navier-Stokes-Korteweg system converge to finite energy solutions to the Euler system \eqref{Euler0}.  

In the case $\kappa^\eps \equiv 0$ and restricting attention to polytropic fluids $p(\rho) \simeq \rho^\gamma$ (with $1<\gamma<3$), Chen and Perepelitsa \cite{CP} first established a convergence result of the form above; they restricted attention 
to the viscosity coefficient $\mu^\eps(\rho) = \eps$ and, by following LeFloch and Westdickenberg's method \cite{LW}, observed that the diffusion term can be
 controled by a~priori estimate derived earlier by Kanel \cite{Kanel} for a different purpose. Next, Huang et al.~\cite{HPWWZ} treated the viscosity functions $\mu^\eps(\rho) = \eps \rho^\alpha$ and polytropic fluids with $\frac{2}{3} < \alpha < \gamma$. More recently, Charve and Haspot~\cite{CH} were the first to tackle the general viscous-capillary problem and established a convergence theorem for polytropic fluids, the viscosity coefficient 
$\mu^\eps(\rho) = \eps \rho$, and the capillarity coefficient $\kappa^\eps(\rho) = \eps^2\rho^{-1}$.  
 
In the present paper, we cover a broad and physically realistic class of viscosity, capillarity, and pressure functions. Our main restriction 
imposes that the capillarity term is ``tame'' with respect to the viscosity. In the theory of nonclassical solutions to hyperbolic conservation laws 
(reviewed in \cite{LeFloch-book}), it has been recognized that the capillarity should be bounded by the square of the viscosity, that is, 
\be
\label{abc}
\kappa(\rho) \lesssim \mu(\rho)^2, 
\ee
for, otherwise, oscillating patterns would be generated in the limit of vanishing capillarity and would overcome the smoothing effect of the viscosity term.  
For instance, this inequality can be justified by considering the behavior of traveling wave solutions (or by numerically computing the vanishing viscosity-capilarity 
limit; cf.~\cite{HL}). When the dispersion effects are dominant, the method adopted in the present paper does not apply and 
the limit, in general, fails to be a weak solution to the associated hyperbolic system: this issue was first investigated by Lax and Levermore \cite{LL1,LL2,LL3} 
in their work on the small dispersion limit of the Korteweg-de Vries equation. 

The inequality \eqref{abc}, while being a general feature of vanishing diffusive-dispersive limits of hyperbolic systems, 
is, in the present setup, valid if $\rho$ is close to a constant, say $\rho \sim 1$. In order to take into account the scaling of the NSK system near the vacuum, 
we now consider (for the purpose of motivating our tame condition) mass density functions $\rho$ that remain close to some constant $r>0$, say. 
By rescaling an (NSK) solution $\rho$ by $r$, we obtain a new solution $\rhob = \rho/ r$ to (NSK), with rescaled viscosity and capillarity coefficients 
$$
\overline \mu(\overline \rho) = r^{-1} \mu(r \overline \rho), 
\qquad \qquad
\kappa(\overline \rho) = r \, \kappa(r \overline \rho). 
$$
When applied to this rescaled system, the condition~(\ref{abc}) becomes 
$
r \, \kappa(r \overline \rho) \lesssim {\mu(r \overline \rho)^2 \over r^2}, 
$
in which $\overline \rho$ remains close to $1$, and $r>0$ is a parameter. By assuming that the implied constant is independent of $r$, this inequality leads us to the condition
\bel{ABC}
\kappa(\rho) \lesssim \frac{\mu(\rho)^2}{\rho^3}
\qquad \quad 
(\rho >0),  
\ee
which is the main restriction required in order to validate the zero viscosity-capillarity limit. The condition~(\ref{ABC}) is thus formulated as  
\be 
\label{defTC}
(TC) \quad 
\begin{cases}
& \kappa(\rho) \lesssim \frac{\mu(\rho)^2}{\rho^3}
\qquad \quad 
(\rho >0), 
\\
& \delta(\eps) \lesssim \eps^2, 
\end{cases}
\ee
and is refered to as the {\bf Tame Capillarity} (TC) condition. 
As explained above, this condition is necessary since, when it is violated, highly oscillating patterns arise in solutions to (NSK) 
and prevent their strong convergence as the viscosity and capillarity approach zero (cf.~again \cite{LeFloch-book}).  

Our second assumption restricts the growth of the viscosity and capillarity coefficients and, specifically, we impose 
the following {\bf Growth Rate} (GR) condition (for all $\rho>0$):  
\bel{technical}
(GR) \qquad  
\aligned
\mu(\rho) \lesssim \rho^{2/3}, 
\\
\rho \kappa' + 5 \kappa \geq 0. 
\endaligned
\ee
The second inequality, essentially, requires the lower bound $\kappa \gtrsim \rho^{-5}$, or else $\kappa$ vanishes identically.
Recall finally that, under the coercivity condition (SC), the energy inequalities associated with (NSK)$_\eps$ yield the following uniform bound for the solutions 
$\rho^\eps, u^\eps$: 
\be
\label{cigogne2}
\begin{split}
& \sup_{t \geq 0} \big( \mathcal{E}_\star[\rho^\eps,u^\eps] (t)+ \widetilde{\mathcal{E}}_\star[\rho^\eps,u^\eps](t) \big) 
 + \iint_{(0,\infty)\times \DD} \mu(\rho) (u^\eps_x)^2 \, dtdx
\\
& + \iint_{(0,\infty)\times \DD} 
\Bigg(
\frac{\mu(\rho^\eps) p'(\rho^\eps)}{(\rho^\eps)^2} (\rho^\eps_x)^2 + \frac{\mu(\rho^\eps) \kappa(\rho^\eps)}{\rho^\eps} (\rho^\eps_{xx})^2 
+ \frac{\mu(\rho^\eps)\kappa(\rho^\eps)}{(\rho^\eps)^3} (\rho^\eps_x)^4 \Bigg)
\, dtdx
\\
&
\lesssim \Escr_\star[\rho^\eps_0,u^\eps_0] + \Escrt_\star[\rho^\eps_0,u^\eps_0],
\end{split}
\ee
in which we now assume that the total energy of the initial data remains uniformly bounded as $\eps \to 0$. 

\begin{theorem}[The zero viscosity--capillarity limit] 
\label{theo:12}
Consider the Navier-Stokes-Korteweg NSK system$_\eps$ posed on the real line $\RR$ (with $\rho_\star \geq 0$) and for polytropic pressure laws 
$p \simeq \rho^\gamma$ with $\gamma \in (1, 5/3]$. Assume that $\kappa^\eps$ and $\mu^\eps$ have the form~(\ref{eq:scaling}),  where 
the viscosity and capillarity coefficients satisfy the tame condition (TC), the strong coercivity condition (SC) and the growth rate condition (GR). 
Consider global-in-time solutions $(\rho^\eps,u^\eps): [0, +\infty) \times \RR \to [0, +\infty) \times \RR$ associated with a family of initial data with uniformly bounded total (physical and effective) energy
$$
\limsup_{\eps\to 0} \Big( 
\Escr_\star[\rho_0^\eps,u_0^\eps] + \Escrt_\star[\rho_0^\eps,u_0^\eps] \Big) < +\infty
$$
and 
satisfying the energy-dissipation inequality (\ref{cigogne2}). 
Then, as $\eps\to 0$, the family of solutions $(\rho^\eps,u^\eps)$ converges 
almost everywhere\footnote{At points where the limit $\rho$ vanishes, the momentum $\rho^\eps u^\eps$ converges toward $\rho u$ 
but the velocity $u^\eps$ need not converge.} 
(after possibly extracting a subsequence) toward a limit $(\rho, u):  \RR_+ \times \RR \to [0, +\infty) \times \RR$ which is a weak solution with finite total energy to the Euler system \eqref{Euler0}. 
 When $\rho_\star = 0$ and the total mass is uniformly bounded at time $t=0$, that is, 
$$
\limsup_{\eps\to 0} \Mscr[\rho_0^\eps] < +\infty, 
$$
then the limit has finite total mass for all times, that is, 
$$
\Mscr[\rho(t, \cdot)] < +\infty, \qquad t \geq 0.  
$$
\end{theorem}

For simplicity in the presentation, our convergence result is stated for polytropic fluids and for plane-symmetric fluids defined 
on the real line, although these restrictions can actually be removed and we refer to \cite{GL2,GL3} for further details. 
Finally, from Theorem~\ref{theo:12}, 
we recover an existence result first established in \cite{LW}.  

\begin{corollary}[Existence theory for the Euler system with finite energy data]
\label{coro-1}
\noindent
Given any initial data $(\rho_0, u_0)$ with finite total energy 
$$
\Escr_\star[\rho_0,u_0] < +\infty 
$$
for some $\rho_\star \geq 0$, 
the corresponding initial value problem associated with the Euler system of polytropic perfect fluids admits a global--in--time solution $(\rho, u): 
\RR_+ \times \RR \to [0, +\infty) \times \RR$, which has 
finite energy for all times, with 
$$
\Escr_\star[\rho,u](t) \leq \Escr_\star[\rho_0,u_0], \qquad t \geq 0.
$$
Furthermore, when $\rho_\star$ vanishes and the total mass is initially finite, then 
$$
\Mscr[\rho(t, \cdot)] = \Mscr[\rho_0] < +\infty, \qquad t \geq 0.
$$
\end{corollary}

Recall that, if the condition (NC) is assumed, the solutions to (NSK) given by Theorem~\ref{theo:11} satisfy the non-cavitating property $\inf_\RR \rho^\eps(t) >0$ (for all $t \geq 0$). In constrast, ``general'' solutions to the Euler system given in Corollary~\ref{coro-1}
may always contain vacuum regions. 
Furthermore, the solutions in Corollary~\ref{coro-1} are shown to satisfy only the conservation of mass and momentum, 
while, in LeFloch and Westdickenberg \cite{LW}, the weak solutions were proven to satisfy all of entropy inequalities (associated with subquadratic test-functions). 


\subsection{Shallow water flows, quantum hydrodynamics, and Boussinesq models}
\label{sec-section1.4} 

First of all, the Navier--Stokes--Korteweg system \eqref{eq:101} describes the isentropic flow of a compressible fluid 
represented by its mass $\rho \geq 0$ and velocity $u \in \RR$, subject to a viscous force $\mu(\rho) u_x$, when the internal energy admits the decomposition $e(\rho) + \kappa(\rho) (\rho_x)^2$. The constitutive behavior of the fluid is determined by prescribing the viscosity function $\mu=\mu(\rho)$ and the capillarity function $\kappa= \kappa(\rho)$, 
as well as the pressure law $p=p(\rho)$ (or, equivalently, the internal energy $e=e(\rho)$).

However, aside from modeling a compressible fluid subject to viscous and capillary forces, the system \eqref{eq:101} also arises in many other physical applications, including in quantum hydrodynamics and in the theory of water waves. Let us present the relevant expressions of the functions $\mu$, $\kappa$, and $p$ for each of these model. 

\begin{itemize}

\item {\bf Polytropic fluids and van der Waals fluids.} The pressure function is classically assumed to be the one of a polytropic perfect fluid $p(\rho) \simeq \rho^\gamma$ with $\gamma \in (1, +\infty)$. In the kinetic derivation of the Navier--Stokes system presented in~\cite{ChapmanCowling}, the viscosity coefficient $\mu$ is a prescribed function of the temperature $T$, specifically $\simeq \sqrt{T}$. 
For a polytropic perfect fluid (with fixed entropy), one thus has $T \simeq \rho^{\gamma-1}$, which leads to the law $\mu(\rho) \simeq \rho^{(\gamma-1)/2}$ for the viscosity. Another classical equation of state which describes complex fluids beyond ideal fluids
is given by the equation of van der Waals (after a standard normalization) 
$p(\rho) = \alpha \, \rho^\gamma \, (3-\rho)^{-\gamma} - 3 \, \rho^2$ (with $\rho < 3$) for some adiabatic exponent $\gamma>1$ and with $\alpha:=8e^{3(\gamma-1) S/8}$ (the constant $S$ representing the entropy). 

\item {\bf Shallow water flows.} 
The Saint Venant model, also called shallow water model, is formally identical to the Euler equations \eqref{Euler0} but corresponds 
to the pressure law $p(\rho) \simeq \rho^2$. The viscous shallow water equation (for instance derived in~\cite{BreschNoble} and also in ~\cite{GP}) corresponds to the coefficients $\mu(\rho) \simeq \rho$ and  $\kappa=0$. 
Finally, in order to include surface tension effects as in~\cite{Marche}, we can take $\kappa(\rho) \simeq 1$.

\item {\bf Quantum hydrodynamics and Bose-Einstein condensates.} It is well--known that the Madelung transform turns a nonlinear Schr\"odinger equation of the form $i \partial_t u - \Delta u =f(|u|^2) u$ 
into the NSK system~(\ref{eq:101}), in which $p=f$, $\mu \equiv 0$, and
$\kappa(\rho) \simeq 1/\rho$. A typical expression for the function $p(\rho)$ is the linear function $\rho$ (for the cubic NLS), but a wealth of more refined models exists such as, for instance, $p(\rho) = A \rho^\nu 
+ B \rho^{2\nu}$; cf.~\cite{KLD} for further details. A closely related model is provided by the quantum Navier--Stokes equation, described in Harvey~\cite{Harvey} and~\cite{BrullMehats}, which takes the 
form~(\ref{eq:101}) with the coefficients $\mu(\rho) \simeq \rho$ and $\kappa(\rho) \simeq 1/\rho$. This model is also discussed  by Hoefer et al. \cite{Hoefer} in connection with Bose-Einstein condensates and
regarded (after transformation) as an extension of Gross-Pitaevskii equation.

\item {\bf Boussinesq model and generalizations.}  It is also instructive to write the system~(\ref{eq:101}) in mass Lagrangian coordinates $(t,y)$ defined by $dy = \rho dx - \rho u dt$ and to 
introduce the new dependent variable $v = 1/\rho$; cf.~Section~\ref{sec:LC} and equations~\eqref{eq:NSKfinal12}. As observed in ~\cite{BDD1}, the Euler equation in Lagrangian coordinates coincides with the Boussinesq equation when $\kappa(\rho) = \rho^{-5}$. 
The standard Boussinesq equation corresponds to $p(\rho) = \rho^2$, but it is possible to consider more general pressure laws; cf.~Bona and Sachs~\cite{BS} for a class of generalized Boussinesq equations. See also  Green and Naghdi~\cite{GN} and Lannes and Bonneton~\cite{LannesBonneton}.  
\end{itemize}

The above examples suggest to consider polynomial (or rational) functions $p$, $\mu$, and $\kappa$. 
Although our framework is much more general, it is thus interesting to indicate the range of application for
 our main results when $p$, $\mu$, and $\kappa$ are power laws of the form 
$$
p(\rho) = p_0 \rho^\gamma, 
\qquad \mu(\rho) = \mu_0 \rho^{\alpha}, \qquad \kappa(\rho) = \kappa_0 \rho^\beta.
$$
First of all, the conditions (NC), (SC), (TC), and (GR) are then equivalent to 
$$
\aligned
(NC) \quad & \mbox{$\alpha < \frac{1}{2}$ or $\beta <-2$},
 \\
(SC) \quad & 2 \alpha -4 < \beta < 2 \alpha - 1,
 \\
(TC) \quad & \beta = 2 \alpha - 3,
 \\
(GR) \quad & \beta \geq -5,
\endaligned
$$
respectively. Observe that the constraint $\mu(0)=0$ imposes $\alpha \geq 0$. It is assumed throughout Theorem~\ref{theo:11} that (SC) holds; furthermore, the four items in this theorem correspond to
$$
\aligned
(i) \quad & \alpha < \frac{1}{2} \ \text{ or } \ \beta <-2,
\\
(ii) \quad &  \kappa_0 = 0;  \qquad && 2\gamma > \alpha,
\\
(iii) \quad &  \kappa_0 = 0;      && 0<\alpha<1,
\\
(iv) \quad &  -4< \beta < -1; \qquad && \beta\geq 2\alpha -3; \qquad &&&\alpha>\frac{2}{3}, 
\endaligned
$$
respectively. The assumptions in Theorem~\ref{theo:12} correspond to 
$$
\beta = 2\alpha -3;  \qquad\beta > -5; \qquad \alpha>\frac{2}{3},  
$$
respectively, so that the latter two inequalities are equivalent to $\alpha >2/3$. 


\subsection{The finite energy method for augmented Euler systems}

We complete this introduction with a sketch of the method of proof developped in this paper and the companion papers \cite{GL2,GL3}. 
Recall that the proposed method is built upon a strategy first introduced by LeFloch and Westdickenberg \cite{LW} 
in order to cope with geometrical effects in flows with radial symmetry or within a nozzle. The present work allows us 
to encompass augmented versions of the Euler system for real compressible fluids, as well as to study singular limit problems.  
In the rest of this section, we especially emphasize this method for the vanishing viscosity-capillarity problem. 

\vskip.15cm 

{\bf 1. Initial data with finite total (physical and effective) energy for augmented Euler systems.} 
Our main assumption is that the initial data of the Euler system \eqref{Euler0} have finite energy, only. When the problem is posed on a compact domain, such as the torus $\TT$ 
we thus assume that the total energy 
$\Escr_\star[\rho_0,u_0]$ of the initial data $(\rho_0, u_0)$ is finite.  
On an unbounded domain such as the real line $\RR$, the asymptotic limit $\rho_\star \geq 0$
of the mass density at infinity must be specified and we must use the normalized energy \eqref{eq:Estar}. 
Furthermore, when $\rho_\star = 0$ is chosen to vanish and the total mass $\Mscr[\rho_0]$ is also finite, then this condition also holds for all times.  
 
Given any augmented version of the Euler system such as the system \eqref{eq:101}, we naturally impose that the augmented total energy is finite at the 
initial time. This energy now takes into account contributions associated with the augmented terms and, specifically for the NSK system, 
the capilarity contributes $\frac{1}{2} \kappa(\rho) \rho_x^2$. 
Importantly, this term has a favorable sign, when the capillarity is positive, as is implied by the physical modeling.   

Furthermore, we observe in this paper that the augmented system \eqref{eq:101} also admits an {\sl effective total energy,} 
obtained by a suitable transformation of the unknowns of the original system. 
The effective velocity $\ut$ defined in \eqref{eq:101} 
is introduced and the corresponding effective energy $\widetilde{\mathcal{E}}_\star[\rho^\eps,u^\eps](t)$ 
contains the term $\frac{1}{2} \mu(\rho)^2 \rho_x^2/\rho^3$, which is now a contribution of the viscosity and, again,
arises with a favorable sign. This effective energy was first used by Bresch and Desjardin for the Navier-Stokes system \cite{BreschDesjardins}.

\vskip.15cm 

{\bf 2. Global-in-time finite energy solutions to augmented Euler systems.}  
Our first task is to establish the existence of global-in-time solutions to the augmented system under consideration, when the initial data have the integrability and regularity properties 
implied by the finite energy condition, which, for the NSK system, takes the form \eqref{sittelle}, only. 
Whenever the cavitation phenomena can be avoided and the mass density $\rho$ remains bounded away from zero,
it is straightforward to define a notion with finite energy weak solutions in the sense of distributions. Yet, in order to establish the existence of weak solutions,
an important structure conditions on the augmented terms is required: 
while the physical energy is naturally dissipative, a condition arises for this property to hold for the {\sl effective energy}
and, in this paper, for the NSK model, we introduce the notion of {\sl strong coercivity,} as will 
presented and investigated in Section~\ref{sec:2}, below. 

However, cavitation usually occurs in solutions to augmented models when very general constitutive laws are considered. For instance, in the NSK model, 
this is the case when the viscosity and the capillarity are too``weak'' near the vacuum; see the condition ~\eqref{behave}, above.
To handle the cavitation phenomena in augmented models, we need a notion of {\sl weak solution with cavitation}.   
For the NSK model, this issue is discussed in the second part of Section~\ref{sec:3}, below. 

\vskip.15cm 

{\bf 3. Higher-integrability property of the pressure.}
Next, in order to analyze the singular limit ($\eps \to 0$, say) when the augmented model formally converges to the original Euler system, 
we assume that the physical and effective energies of the augmented model are {\sl uniformly} bounded with respect to the parameter $\eps$. 
Our first task is to derive several additional {\sl higher-integrability properties} of the solutions to the augmented system, which 
allow us to get a better control on the solutions. Importantly, the bounds should be {\sl uniform} as $\eps \to 0$. 
Standard parabolic-type or dispersive-type bound simply blow-up when $\eps \to 0$, and from now on
we must exhibit additional structure from the Euler system.   

Our first technique extends an argument in De~Lellis, Otto, and Westdickenberg \cite{DOW} 
(for scalar conservation laws) and LeFloch and Westdickenberg \cite{LW} (for the Euler system), and 
combines the conservation laws for the mass and momentum, as follows. For any system of the form  
\bel{eq:102-two} 
\aligned
\rho^\eps_t + (\rho^\eps u^\eps)_x =& \, X^\eps[\rho^\eps, u^\eps]_x,
\\
(\rho^\eps u^\eps)_t + (\rho^\eps (u^\eps)^2 + p(\rho^\eps) \big)_x =& \, Y^\eps[\rho^\eps,u^\eps]_x, 
\endaligned 
\ee
with augmented terms denoted by $X^\eps[\rho^\eps, u^\eps]$ and $Y^\eps[\rho^\eps, u^\eps]$, we can 
introduce a function $h^\eps$ by setting 
\be
h^\eps_x := \rho^\eps, \qquad \qquad h^\eps_t := - \rho^\eps u^\eps + X^\eps[\rho^\eps, u^\eps], 
\ee
and then write 
\be
(\rho^\eps u^\eps h^\eps)_t + \big( \rho^\eps (u^\eps)^2 h^\eps + p(\rho^\eps) h^\eps \big)_x 
= \rho^\eps p(\rho^\eps) + \rho^\eps u^\eps X^\eps[\rho^\eps, u^\eps] + h^\eps Y^\eps[\rho^\eps, u^\eps]. 
\ee
This identity is used as follows;  cf.~Section~\ref{hipt}, below, for the NSK system. 
By multiplying this identity by a positive test-function $\theta=\theta(t,x)$ and after a suitable integration argument, 
we obtain the uniform estimate on the spacetime integral 
\be
\iint \Big(  \rho^\eps p(\rho^\eps) + \rho^\eps u^\eps X^\eps[\rho^\eps, u^\eps] + h^\eps Y^\eps[\rho^\eps, u^\eps]
\Big) \, \theta \, dtdx. 
\ee

At this juncture, our main observation 
is that the main term $\rho^\eps p(\rho^\eps) \theta$ is non-negative, and this approach eventually leads us to the spacetime estimate
\bel{eq:19}
\rho^\eps p(\rho^\eps) \in L^1_\loc, 
\ee
while the contributions from the augmented terms turn out to be controlable by the energy-type estimates already established. For instance, for the Navier-Stokes model which contains the viscosity term $Y^\eps[\rho^\eps, u^\eps] = \mu(\rho^\eps) u^\eps_x$, 
we check below that a mild condition on the growth of the viscosity guarantees that the $L^1$ average of $Y^\eps[\rho^\eps, u^\eps]$ is controled by the initial (physical and effective) energy. 
Furthermore, as observed by LeFloch and Westdickenberg \cite{LW}, this argument applies in radial symmetry and leads to a uniform estimate valid even at the center of symmetry. 

\vskip.15cm 

{\bf 4. Higher-integrability property of the velocity.}
A better integrability property for the fluid velocity must be derived next an this is achieved with suitably chosen mathematical entropies of the Euler system, 
following an idea in~Lions et al.~\cite{LPT}. 
Recall that the family of weak\footnote{That is, entropies vanishing on the vacuum.} entropy pairs is generated
by an entropy kernel, denoted below by $\chi =\chi(\rho,u;v)$, and an 
entropy flux kernel, denoted by $\sigma=\sigma(\rho,u;v)$. Specifically, for 
any continuous function $\psi=\psi(v)$ (with subquadratic growth, say) 
we can introduce  
\bel{eq:4200-2}
\eta^\psi(\rho,u) := \int_\RR \chi(\rho,u;v) \, \psi(v) \, dv, 
\quad 
\qquad q^\psi(\rho,u) := \int_\RR \sigma(\rho,u;v) \psi(v) \,dv
\ee
and, for any smooth solution to an augmented Euler system, derive additional balance laws of the form   
\bel{eq:20}
\del_t \eta^\psi(\rho^\eps,u^\eps) + \del_x  q^\psi(\rho^\eps,u^\eps) = Z^\eps[\rho^\eps, u^\eps]. 
\ee
Here, the right-hand side $Z^\eps[\rho^\eps, u^\eps]$ vanishes for smooth solutions to the Euler system
and in addition, for solutions to the augmented model, can also be controled by the energy-type estimates already established. 

The second higher-integrability estimate is now obtained, by  integration of \eqref{eq:20}, and by observing that we control the entropy flux  
\bel{eq:vel1}
\sup_{x} \int_0^T q^\psi(\rho^\eps,u^\eps)(t,x) \, dt 
\ee
in terms of the total entropy and a spacetime contribution in $Z^\eps[\rho^\eps, u^\eps]$, both latter terms being already controled by the existing estimates. 
Here, our main observation 
is that, for the family of functions $\psi = v |v|$, the inequality \eqref{eq:20}
has a {\sl non-negative entropy flux} $q^\psi(\rho,u)$. The flux in  \eqref{eq:vel1} grows typically like $|u|^{3}$ in terms of the velocity. 

\vskip.15cm 

{\bf 5. Young measures for finite energy weak solutions.}
Equipped with the estimates \eqref{eq:19} and \eqref{eq:vel1}
we are then in a position to introduce a Young measure, say $\nu: \RR_+ \times \RR \to \text{Prob}(\RR_+ \times \RR)$, in order 
to represent all weak limits of expressions like $f(\rho^\eps, u^\eps)$.
Following \cite{LW}, we rely first on the higher-integrability estimate \eqref{eq:19} for the mass density 
and establish the weak convergence 
\bel{eq:nu1}
f(\rho^\eps, u^\eps) \to\la \nu, f \ra := \iint_{\RR_+ \times \RR} f(\rho,u) \, d\nu 
\ee
for all continuous function $f=f(\rho,u)$ satisfying the growth condition $|f(\rho, u)| \lesssim f_0(\rho) \, \rho p(\rho)$ with $\lim_{\rho \to + \infty} f_0(\rho) = 0$. Next, we take into account the higher-integrability estimate  \eqref{eq:vel1} for the velocity
and we check that we can allow a velocity behavior of the form $f_1(u) \, |u|^{3}$ with $\lim_{|u| \to + \infty} f_1(u) = 0$. 

At this juncture, we point out that the arguments above apply to the Cauchy problem posed on the real line, but need some adaptation to apply to the torus. Here,  we can also rely on 
a property of {\sl propagation of equi-integrability for the velocity,} first proposed in \cite{LW}.  
Here, the basic strategy is to integrate the entropy balance law \eqref{eq:20} with, in \eqref{eq:4200-2}, functions $\psi$ with suitably chosen support in the velocity variable. 

\vskip.15cm 

{\bf 6. Reduction with finite energy Young measures for real compressible fluids.} Our next task is to derive and analyze 
Tartar's commutation relation \cite{Tartar1}
 satisfied by the Young measure $\nu$ for every pair of mathematical entropies, that is, for all 
$\psi_1, \psi_2$ we rely on the div-curl lemma and establish that 
\bel{eq:30} 
\la \eta^{\psi_1}  q^{\psi_2} -  q^{\psi_1} \eta^{\psi_2} \ra 
= 
\la \eta^{\psi_1} \ra \la q^{\psi_2} \ra -  \la q^{\psi_1} \ra \la \eta^{\psi_2} \ra  
\ee 
at almost every point $(t,x)$. The higher-integrability properties above are essential in this derivation, in order to allow all functions 
$\psi_1, \psi_2$ with subquadratic growth at infinity. 

At this juncture, a major difficulty is to deduce from \eqref{eq:30} 
that $\nu$ reduces to a Dirac mass at each point $(t,x)$, at least away from the vacuum, which is equivalent to 
the strong convergence of $\rho^\eps$ and $\rho^\eps u^\eps$.
This is done by exhibiting some {\sl unbalance of regularity} between the two sides of \eqref{eq:30}. 
For polytropic perfect fluids, this was done in the references cited above and generalized 
in \cite{LW} to possibly unbounded Young measures with finite energy. 
The generalization to real fluids is presented in the follow-up paper \cite{GL2}.

\vskip.15cm 

{\bf 7. Global-in-time finite energy solutions to the Euler system.}
The above steps have thus allowed us to fully validate the passage to the limit $\eps \to 0$. 
The relevant notion of a finite energy solutions to the Euler equations, first introduced in \cite{LW}, yields that, 
in particular, such a solution $(\rho, u)$ satisfies the bound  
\bel{eq:finiteE}
\Escr[\rho(t),u(t)] \, dx < + \infty, \qquad t \geq 0.
\ee
In the present work where we take physical viscosity as well as capillarity terms into acount, the entropy inequalities need not hold, however.  

\vskip.15cm 

{\bf 8. Subcritical and critical scalings.}
Our theory covers the regime where the capillarity is dominated by the viscosity, in the sense of the tame condition \eqref{defTC}.
In the subcritical scaling
\be
\frac{\delta(\eps)}{\epsilon^2} \overset{ \epsilon \rightarrow 0 }{\longrightarrow} 0,
\ee
the capillarity terms are ``negligible'' in the limit and it is expected that the solutions we obtained in Corollary~\ref{coro-1} 
satisfy all entropy inequalities (with sub-quadratic growth in the velocity variable). This property can be checked for traveling wave solutions, at least, along the lines of \cite{BL1,MP}.  
On the other hand, the most interesting regime from the mathematical and physical standpoints, arises in the {\sl critical scaling,}
when the diffusive and dispersive effects within the augmented Euler model are ``kept in balance'', in the sense that  
\be
\delta(\eps) = \alpha \, \eps^2 \qquad \text{ ($\alpha$ fixed)}. 
\ee
Then, dispersive terms generate genuine oscillations which (in the limit $\eps \to 0$) drive the effective dynamics of ``dispersive shock waves'', so that 
a different selection mechanism may be observed and shock waves may fail to satisfy standard entropy conditions \cite{BL1,BL5,JosephLeFloch}. 
However, in the present paper, since the pressure function is assumed to satisfy the genuine nonlinearity condition \eqref{GNLc}, we again conjecture that the entropy inequalities are satisfied by the solutions constructed in Corollary~\ref{coro-1}. Again, this property can be checked for traveling wave
solutions, at least.

The rest of this paper is organized as follows. In Section~\ref{sec:2}, we consider the Navier--Stokes--Korteweg system,
discuss basic algebraic properties, and introduce our strong coercivity condition. 
In Section~\ref{sec:3}, we establish an existence theory for the NSK system by constructing weak solutions when 
the initial data have finite energy and the viscosity and capillarity functions satisfy certain mild conditions.
Next, in Section~\ref{sec:4}, we establish our two higher--order integrability properties for the Navier-Stokes-Korteweg system 
and we conclude with the strong convergence of weak solutions to the NSK system
toward weak solutions to the Euler as the viscosity and capillarity tend to zero.


\section{Conservation laws and the strong coercivity condition} 
\label{sec:2}

\subsection{Derivation in Lagrangian coordinates}
\label{sec:LC}

The Navier-Stokes-Korteweg system is derived (in mass Lagrangian coordinates) as follows (cf.~\cite{GG} for details). 
We denote by $(t,y) \mapsto \chi(t,y)$ the so-called mass Lagrangian map, defined so that the integral $\int_a^b \chi(0,y) \,dy$
represents the total mass which was initially located in the interval $[a,b]$ and, moreover,  
the mass initially located at some point $y \in \DD$ has moved to $\chi(t,y) \in \DD$ at the time $t \geq 0$. From this map, we define 
the specific volume $v=1/\rho$ (or equivalently the density $\rho$) together with the velocity $v$ by 
$$ 
u := \chi_t, \qquad \quad v := \chi_y.
$$
For the sake of simplicity in the notation, we keep the same notation for constitutive functions expressed in Lagrangian or in Eulerian coordinates. 
We proceed by prescribing an internal energy function of the form $e = e(v, v_y)$, and we postulate that the following action (on a time interval $[0,T]$)
\bel{eq:action}
{\mathcal J}(\chi) 
:= \iint_{[0,T] \times \DD} \Big(e(v, v_y) - u^2 /2 \Big) \, dtdy
= \iint_{[0,T] \times \DD}  \Big(e(\chi_y, \chi_{yy}) - \chi_t^2/ 2 \Big) \, dtdy
\ee
is formally extremal among all such maps $\chi$. It is easy to derive the Euler-Lagrange equation associated this variational problem, 
namely 
\bel{eq:equachi}
\chi_{tt} + \Bigg( - { \del e \over \del \chi_y }(\chi_y, \chi_{yy}) 
          + \Bigg( { \del e \over \del \chi_{yy} }(\chi_y, \chi_{yy})
 \Bigg)_y \, \Bigg)_y = 0. 
\ee
Next, by observing that $v_t = \chi_{yt} = u_y$ and introducing the pressure function 
$$
P(v,v_y, v_{yy}) := - { \del e \over \del v }(v,v_y) + \Bigg( { \del e \over \del v_y }(v,v_y) \Bigg)_y,   
$$ 
we deduce that the unknown state variables $u, v$ satisfy the following {\bf Euler-Korteweg system}  
\be
\aligned 
v_t - u_y &= 0,               
\\
u_t + P(v,v_y, v_{yy})_y &= 0. 
\endaligned 
\ee
This system includes the effects of the propagation of waves in the fluid described by the pressure function $P$
as well as the effects of the capillarity which is modeled by the internal energy function $e$.

In addition, by prescribing a viscosity function $\nu=\nu(v)$, we arrive at the {\bf Navier--Stokes--Korteweg model} in mass Lagrangian coordinates 
\bel{eq:NSK-Lagrange0} 
\aligned 
v_t - u_y & = 0,
\\
u_t + P(v,v_y, v_{yy})_y & = (\nu(v) \, u_y )_y. 
\endaligned 
\ee
Observe that the local energy $E(v,u,v_y) = e(v,v_y) + u^2/2$ satisfies the additional conservation law 
\bel{eq:NSK-Lagrange1} 
E(v,u,v_y)_t + \bigl(P(v,v_y, v_{yy}) \, u\bigr)_y 
   =  \Big( u_y \, { \del e \over \del v_y }(v,v_y) \Big)_y 
      + \bigl( \nu(v) \, u \, u_y \bigr)_y - \nu(v) \, u_y^2. 
\ee 

It remains to comment about the internal energy function. 
A standard choice made in physics (in phase dynamics, in particular) is a quadratic dependency of $e$ with respect to $v_y$, that is, 
\bel{eq:e}
e(v,v_y) = e(v) + \lam(v) \, {v_y^2 \over 2},
\ee 
where $\lam=\lam(v)$ is refered to as the capillarity coefficient.  Observe that linear terms cannot arise, due to
the invariance of the physical laws by the transformation $y \mapsto -y$. 
Consequently, the pressure $P$ splits into a function of $v$ and a capillarity term, as follows: 
$$ 
p(v) = - e' (v), 
\qquad \qquad
P(v,v_y, v_{yy}) =  p(v) - \lam'(v) \, {v_y^2 \over 2} + (\lam(v) \, v_y )_y. 
$$ 
Hence, for the constitutive law \eqref{eq:e}, the Navier--Stokes--Korteweg system takes the form: 
\bel{eq:NSKfinal12}
\aligned 
& v_t -  u_y = 0,  
\\
& u_t + p(v)_y 
      =           \bigl(\nu(v) \, u_y \bigr)_y, 
+\Big( \lam'(v) \, {v_y^2 \over 2} - \bigl(\lam(v) \, v_y \bigr)_y \Big)_y 
\endaligned 
\ee
while the associated energy balance equation reads 
\bel{eq:NSKfinal3}
\aligned 
& \Bigg(e(v) + {u^2 \over 2} + \lam(v) \, {v_y^2 \over 2}\Bigg)_t 
  + \bigl(p(v) \, u \bigr)_y 
\\
& = 
        \bigl( \nu(v) \, u \, u_y \bigr)_y  - \nu(v) \, u_y^2 
+
\Bigg(u \, \Big( {\lam'(v)\over 2} \, v_y^2  
         - \bigl(\lam(v) \, v_y \bigr)_y \Big)
         + u_y \, \lam(v)\, v_y \Bigg)_y.
\endaligned 
\ee 
The Lagrangian-Eulerian transformation $t \mapsto y(t,x)$, defined by $y_t = -\rho u$ and $y_x = \rho$, allows us to derive the Eulerian formulation
\eqref{eq:101} from the Lagrangian formulation above. It is easy to check that the relation
\be
v \, \nu(v) = \mu(1/v), 
\qquad 
v^5 \lambda(v) := \kappa(1/v)
\ee
holds between the viscosity and capillarity coefficients in Lagrangian and Eulerian coordinates.


\subsection{Local balance laws} 
\label{sec:Section23}

\subsubsection*{Conservation law for the mass} 

We now record several elementary but fundamental properties of the Navier--Stokes--Korteweg system. Observe that the first equation in \eqref{eq:101}, that is
\bel{eq:400}
\rho_t + (\rho u)_x = 0,
\ee
simply expresses the local conservation of the mass density.

\subsubsection*{Conservation law for the momentum} The second equation in~(\ref{eq:101}), i.e. 
\bel{eq:401} 
(\rho u)_t + \Big(\rho u^2 + p(\rho)  
- \mu(\rho) u_x  
- \rho \kappa(\rho) \rho_{xx} - \frac{1}{2} (\rho \kappa'(\rho) - \kappa(\rho)) \rho_x^2 \Big)_x 
=0 
\ee
expresses the local conservation of the momentum $\rho u$ and, in view of the mass equation above, 
has the equivalent form 
\bel{eq:402}
u_t + u u_x + \frac{p'(\rho)}{\rho} \, \rho_x 
= \frac{1}{\rho} \, \big( \mu(\rho) u_x \big)_x + \Big( \kappa(\rho) \rho_{xx} + \frac{1}{2} \kappa'(\rho) (\rho_x)^2 \Big)_x.
\ee
Observe that all the terms above can be given a conservative form (for instance, for the pressure term by introducing the function $k(\rho):=\int^\rho \rho^{-1} p'(\rho) d\rho$), except the viscosity term which contains an extra  factor $1/\rho$. Hence, in the limit of vanishing viscosity and capillarity and for {\sl weak} solution, the expression $u_t + (u^2/2)_x + (k(\rho))_x$ should {\sl not} be expected to vanish in the sense of distributions.

\subsubsection*{Balance law for the energy}
\label{sec:energy}

We already introduced the local energy $E= E[\rho,u]$ of the NSK system by 
$$
E = \frac{1}{2} \rho u^2 + \rho e(\rho) + \frac{1}{2} \kappa(\rho) \rho_x^2
$$
with $e'(\rho) = \frac{p(\rho)}{\rho^2}$. Let us now define the {\bf local internal forces} $F= F[\rho,u]$
by  
$$
F := - p(\rho) + \mu(\rho) u_x + \rho \kappa(\rho)\rho_{xx} + \frac{1}{2} \big( \rho \kappa'(\rho) - \kappa(\rho) \big) \, (\rho_x)^2
$$
and the {\bf local energy dissipation} $D= D[\rho,u]$ as  
$$
D := \mu(\rho) u_x^2. 
$$
Then, the local energy balance law reads 
\be
\label{eq:23} 
E_t + \Big( u E - u F + \kappa(\rho) \rho \rho_x u_x \Big)_x = -D \leq 0.
\ee
Recall that Dunn and Serrin in~\cite{DunnSerrin} refer the term $\kappa(\rho) \rho \rho_x u_x$, above, as the {\bf interstitial work.}  

\subsubsection*{Effective velocity}

Given any $\omega \in \RR$, 
we propose here to define the {\bf $\omega$-effective velocity} 
$$
\ut^\omega = u + \omega \, \frac{\mu(\rho)}{\rho^2} \rho_x.
$$
By taking 
$\omega=1$, we recover an expression introduced first in Bresch, Desjardins and Lin~\cite{BreschDesjardins} and used by Mellet and Vasseur~\cite{MelletVasseur}, while J\"ungel~\cite{Jungel} used $\omega=1/2$.

\subsubsection*{$\omega$-Effective NSK system}

We observe here (cf.~the derivation at the end of this section that)  
the pair $(\rho,\ut^\omega)$ solves a new system which has essentially the same 
algebraic structure to the one of the NSK system, that is,  
\bel{eq:27}
\aligned
\rho_t + (\rho \ut^\omega)_x & =  \left( \omega \frac{\mu(\rho)}{\rho} \rho_x \right|_x, 
\\
(\rho \,\ut^\omega)_t + \big( \rho (\ut^\omega)^2 + p(\rho) \big)_x 
& = \, M^\omega[\rho,u]_x 
+ K^\omega[\rho]_x,  
\\
M^\omega[\rho,u] :=& {\mu(\rho) \over \rho} \, \Big( 
(1-\omega) \, \rho \, \ut^\omega_x 
+ \omega  \, \rho_x  \, \ut^\omega\Big), 
\\
K^\omega[\rho] :=&\,\rho \, \kappat^\omega (\rho) \rho_{xx} 
+ \frac{1}{2} \, \Big (\rho {\kappat^\omega}'(\rho) - \kappat^\omega(\rho)   \Big) \, \rho_x^2,   
\endaligned
\ee
which we propose to refer to as the {\bf $\omega$-effective Navier--Stokes--Korteweg system} and 
in which we have introduced the following {\bf $\omega$-effective capillarity} coefficient 
\bel{eq:28}
\kappat^\omega := \kappa - \omega (1-\omega) \frac{\mu^2}{\rho^3}.
\ee
Observe that the structure of the capillarity terms is exactly {\sl preserved,} while the viscosity is ``split'' between the mass and momentum equations, and 
the most important outcome of this transformation is that the mass equation has now gained a diffusion term. 

The relevant range for $\omega$ appears to be the interval $[0,1]$ for, otherwise, the system is not parabolic; moreover, when $\omega 
\in (0,1)$, this effective system is uniformly parabolic. It is natural also to choose $\omega$ so that $\kappat^\omega \geq 0$ ---which 
always holds if $\omega$ equals $0$ or $1$, or if $\omega$ is sufficiently close to $0$ or $1$ and the reverse inequality in \eqref{ABC} is assumed.  
In the present paper, the choice $\omega = 1$ will play a central role , and we set for the rest of this paper
\bel{eq:choicelambda}
\ut := \ut^1, \qquad \quad
\kappat := \kappat^1. 
\ee


\subsubsection*{Balance law for the $\omega$-effective energy}

For every  $\omega \in [0,1]$, we introduce the {\bf local effective energy} 
$$
\Et^\omega=\Et^\omega[\rho,u] := \frac{1}{2} \rho (\ut^\omega)^2 + \rho e(\rho) + \frac{1}{2} \kappat^\omega(\rho) \rho_x^2, 
$$
and the {\bf local effective energy dissipation} 
\bel{eq:eed}
\aligned
\Dt^\omega = \Dt^\omega[\rho,\ut]
 := \, 
& (1- \omega) \, \mu(\rho) (\ut^\omega_x)^2 
+ \omega \, 
\frac{\mu(\rho)}{\rho^2} p'(\rho) (\rho_x)^2  + \omega \, \frac{\mu(\rho)}{\rho} \kappa(\rho) 
\Big( 
(\rho_{xx})^2 + \zeta(\rho) \rho_x^4 \Big),  
\endaligned
\ee 
with $\zeta=\zeta(\rho)$ being given by 
(the rather involved expression below arising first in the calculations) 
\bel{eq:zeta} 
\aligned
\zeta :=&  {\rho \over \mu \kappa} \Bigg(
\left(\frac{\mu}{\rho^2}\right)' {1 \over 2} \big(\rho \kappa'(\rho) - \kappa(\rho) \big) 
 -  {1 \over 3}
\left( 
\left(\frac{\mu(\rho)}{\rho^2}\right)' \rho \kappa(\rho) \right)'  
 -  {1 \over 3} \left( \frac{\mu(\rho)}{\rho^2} {1 \over 2}  \big( \rho \kappa'(\rho) - \kappa(\rho) \big) \right)' \Bigg) 
\\
= & - {1 \over 3} \left( {1 \over 2} {\kappa'' \over \kappa}  + \left(\frac{\mu}{\rho}\right)'' {\rho \over \mu}\right). 
\endaligned
\ee 
By a tedious calculation, we can check that 
\bel{effectiveenergy}
\aligned
\Et^\omega_t 
& + \Big( 
\ut^\omega\, \Et^\omega - \ut^\omega \Ft^\omega + \omega \, \Gt^\omega\Big)_x  
=  - \Dt^\omega, 
\endaligned
\ee
in which the {\bf effective internal forces} $\Ft^\omega= \Ft^\omega[\rho,\ut]$ are  defined by 
$$
\Ft^\omega
 := - p(\rho) + (1-\omega) \, \mu(\rho) \ut_x + \rho \kappat^\omega(\rho)\rho_{xx} + \frac{1}{2} \big( \rho {\kappat^\omega}'(\rho) - \kappat^\omega(\rho) \big) \, (\rho_x)^2 
$$
and the {\bf $\omega$-effective  interstitial work}  by 
\be
\aligned
\Gt^\omega :=  {\mu \over \rho} \big(( \rho e)' + \ut^2/2 \big) \rho_x 
 + \kappat^\omega(\rho) \rho \rho_x \ut_x 
           - \frac{\mu(\rho)}{\rho^2} \rho \kappa(\rho)   \rho_x \rho_{xx}  
+ 
  {1 \over 3}  \, 
\left( 
\left( {\mu(\rho) \over \rho} \right)' \kappa(\rho)  
      - {\mu(\rho) \over \rho} \kappa'(\rho) \right) 
 \rho_x^3. 
\endaligned
\ee  

At this juncture, we observe that, in \eqref{eq:zeta}, the function $\zeta$ need not be non-negative, so that $\Dt^\omega[\rho,\ut]$ need not be non-negative ---in contrast with the physical dissipation $D[\rho,u]$ which is always non-negative.  Our ``strong coercivity'' condition defined below will ensure that $\Dt^\omega[\rho,\ut]$ is {\sl non-negative in average.}  
Finally, when $\omega$ is chosen to be unit, we shorten our notation and write 
\be
\Et := \Et^1, \qquad \Dt := \Dt^1, \qquad \Ft := \Ft^1, \qquad \Gt := \Gt^1.
\ee


\subsubsection*{Derivation of an effective NSK system}

The derivation of \eqref{eq:27} from \eqref{eq:101}, in principle, is an elementary matter but may be particularly tedious
if a brute force calculation is attempted: 
we propose here a rather natural derivation of this important identity. First of all, the treatment of the mass equation is trivial from the definition 
of the effective velocity, while to handle the momentum equation 
we substract \eqref{eq:101} from \eqref{eq:27} and, therefore, we solely need to establish the identity 
$$
\aligned
& (\omega \rho \, m_x)_t 
+ \Big( 2 \omega \rho u  m_x + \omega^2 \rho m_x^2         \Big)_x 
\\ 
& = \, \Big( \omega m_{xx} \mu
-\omega \, \mu \,  (u + \omega  m_x)_x
+ \omega \, \rho m'\rho_x \, (u + \omega  m_x) \Big)_x 
- \Big( \omega (1 - \omega) \,\rho \,  n \, \rho_{xx} 
+ \omega (1 - \omega)  \frac{1}{2} \, \Big (\rho \, n' - n \Big) \rho_x^2 \Big)_x. 
\endaligned
$$
Here, we have introduced the notation $m=m(\rho)$ and 
$\frac{\mu}{\rho^2} \rho_x =: m_x$, as well as $n:= \mu^2 / \rho^3$, and we have been able to cancel out some capillarity-related terms. 
By defining $\omega\rho m' =: q'$ and writing 
$$
(\omega \rho \, m_x)_t = q_{xt} = q_{tx} 
= - \big( q' (\rho u)_x \big) _x
= - \big( \omega \rho m' (\rho u)_x \big)_x, 
$$
and then observing that $\omega$ should be treated as a parameter, we see that the above identity splits into two distinct identities 
$$
\aligned
\big( - \omega \rho m' (\rho u)_x \big) _x
+ \Big( 2  \omega \rho u  m_x \Big)_x  
=&  \, \Big(  \omega \Big( 
m_{xx} \mu
- \mu \,  u_x
+ \rho m'\rho_x \, u \Big) 
\Big)_x 
-  \Big(  \omega \Big(
\,\rho \,  n \, \rho_{xx} 
+ \frac{1}{2} \, \Big (\rho \, n' - n \Big) \rho_x^2 \Big)
\Big)_x,
\\
\Big(  \omega^2 \rho m_x^2         \Big)_x 
=& \, \Big(  \omega^2 \Big(
- \mu \,  m_{xx} + \rho m'\rho_x \, m_x \Big)
\Big)_x 
+ \Big( \omega^2 \Big(
\,\rho \,  n \, \rho_{xx} 
+ \frac{1}{2} \, \Big (\rho \, n' - n \Big) \rho_x^2 \Big)
\Big)_x. 
\endaligned
$$
Finally, after removing one derivative in $x$ in each term and getting rid of $\omega$ while 
splitting the first equation into terms that depend or are independent of $u$, 
it is sufficient to check the following three identities: 
$$
\aligned
- q' (\rho u)_x + 2 \rho u  m_x 
= & 
- \mu \,  u_x
+ \rho m'\rho_x \, u  
\\
0
=&  \, m_{xx} \mu 
-  \rho \,  n \, \rho_{xx} 
- \frac{1}{2} \, \big (\rho \, n' - n \big) \rho_x^2,
\\
 \rho m_x^2 
=& \, - \mu \,  m_{xx} + \rho m'\rho_x \, m_x 
+ \,\rho \,  n \, \rho_{xx} 
+ \frac{1}{2} \, \Big (\rho \, n' - n \Big) \rho_x^2, 
\endaligned
$$
which indeed, in view of our definition of $m, n,q$, follows easily for {\rm arbitrary} functions $\rho, u$. 


\subsection{The strong coercivity condition}

We are now in a position to state several fundamental uniform estimates and state our main assumption relating the behavior of 
the nonlinear viscosity and capillarity coefficients. We are interested in the mass and energy equations which do provide us with 
non-negative functionals (possibly after a certain normalization), while the momentum has an indefinite sign and cannot be used to derive an uniform estimate.  

\subsubsection*{The Cauchy problem on the torus} 
 
Integrating the mass equation~\eqref{eq:400} in space and time gives us 
\be
\Mscr[\rho](t)  := \int_\TT \rho(t,x) \, dx = \int_\TT \rho(0,x) \, dx = \Mscr[\rho_0] =: \Mscr_0, 
\qquad t \geq 0,
\ee
which is a constant independent of time.  On the other hand, by defining the {\bf total energy} and {\bf total energy dissipation} by
$$
\Escr[\rho,u](t) := \int_\Tbb E(t,x) \, dx, \qquad
\qquad \mathscr{D}[\rho,u](t) = \int_\Tbb D(t,x)\, dx,
$$
we see that the energy equation~\eqref{eq:23} yields 
\be
\label{cormoran1}
\Escr[\rho,u](t) + \int_0^t \mathscr{D}[\rho,u](s) \, ds 
= \Escr[\rho_0,u_0] =: \Escr_0,
\ee
in which $D[\rho,u]$ and thus $\mathscr{D}[\rho,u]$ are {\sl non-negative.} 

Next, by defining the {\bf total effective energy} and {\bf total effective energy dissipation} by
$$
\Escrt[\rho,u](t) := \int_\Tbb \widetilde{E}(t,x) \, dx, 
\quad
 \qquad \widetilde{\mathscr{D}}[\rho,u](t) = \int_\Tbb \widetilde{D}(t,x)\, dx,
$$
in which, by \eqref{eq:eed},  
\bel{eq:eed-omega1}
\aligned
\Dt
 = \Dt[\rho,\ut] :=&  
\frac{\mu(\rho)}{\rho^2} p'(\rho) \rho_x^2  + \frac{\mu(\rho)}{\rho} \kappa(\rho) \Big( 
\rho_{xx}^2 +\zeta(\rho)\rho_x^4 \Big),  
\endaligned
\ee
we obtain 
\bel{cormoran2}
\Escrt[\rho,u](t) + \int_0^t \widetilde{\mathscr{D}}[\rho,u](s)\,ds 
=
\Escrt[\rho_0,u_0] =: \Escrt_0.
\ee 

Observe that, contrary to what happens with the physical energy, the term $\widetilde{\mathscr{D}}[\rho,u]$ in the effective energy balance law need not be non-negative.
The identity \eqref{cormoran2} 
is going to play a central role in our theory, as it provides us with an essential control of the second-order derivative of the mass density, 
that is, $\rho_{xx}$ ---{\sl provided} we can ensure that the effective dissipation remains uniformly positive. 
More precisely, in this paper we propose the following {\bf strong coercivity} (SC) condition: 
for some constant $C_0>0$ and any smooth function $\rho= \rho(x) >0$, 
\bel{eq:203}
(SC) \qquad \int_\Tbb \widetilde{D}[\rho]\, dx 
\geq C_0 
\int_\Tbb \left( \rho_{xx}^2 + {\rho_x^4 \over \rho^2} \right) \, {\mu(\rho)\kappa(\rho) \over \rho} \, dx.
\ee
Clearly, it would be sufficient to assume, for instance that the function $\zeta$ is positive and bounded below by 
$\rho^{-2}$ (up to a multiplicative constant), but in fact it is not necessary that $\zeta$ be positive. 
We refer to Section~\ref{sec:SCC}, below, for the derivation of sufficient conditions on $\mu$ and $\kappa$ guaranteeing that this coercivity inequality holds. 
For instance, it does hold when $\mu(\rho) = \rho$ and $\kappa(\rho) = \rho^\alpha$ with $\alpha \in (-2, 1)$. Observe that the 
term $\frac{\mu(\rho)}{\rho^2} p'(\rho) (\rho_x)^2$ has a different scaling in comparison to the terms $\rho_x^4$ and $\rho_{xx}^2$ and, although
it has a favorable sign, we cannot take advantage of it when proving \eqref{eq:203}. 


\subsubsection*{The Cauchy problem on the real line}

To deal with the problem posed on the real line, we need to introduce a renormalization based on the limit $\rho_\star$ at infinity and, as already stated in the introduction, 
we need to introduce
$$
\aligned
E_\star[\rho,u] 
:=& 
\frac{1}{2} \rho u^2 + \rho e(\rho) - \rho_\star e(\rho_\star) - (\rho e(\rho))'(\rho_\star)(\rho-\rho_\star) + \frac{1}{2} \kappa(\rho) \rho_x^2,
 \\
\Et_\star[\rho,u] :=&
 \frac{1}{2} \rho (\ut)^2 + \rho e(\rho) - \rho_\star e(\rho_\star) - (\rho e(\rho))'(\rho_\star)(\rho-\rho_\star) + \frac{1}{2} \kappa(\rho) \rho_x^2.
\endaligned 
$$
Thanks to the hyperbolicity condition~\eqref{hyperc}, the function $\rho \mapsto \rho e(\rho) - \rho_\star e(\rho_\star) - (\rho e(\rho))'(\rho_\star)(\rho-\rho_\star)$ is non-negative and convex, 
so that $E_\star[\rho,u]$ and $\Et_\star[\rho,u]$ are non-negative. 

Upon defining
$$
\Escr_\star[\rho,u](t) := \int_\RR E_\star[\rho,u] \, dx,
 \quad
\qquad \Escrt_\star[\rho,u](t) := \int_\RR \widetilde{E}_\star[\rho,u] \, dx,
$$
the identities~(\ref{cormoran1}) and~(\ref{cormoran2}) remain valid 
 but the integration domain
is changed to $\RR$ and the normalized energies are used (but 
the dissipation terms do not need to be renormalized).

Recall that we use the notation $\Escr_\star$ for the torus as well as for the real line, with the convention that $\rho_\star = 0$ in the former. The strong coercivity condition (SC) associated with the Cauchy problem posed on the real line is stated as follows: 
for some constant $C_0>0$ and any smooth function $\rho= \rho(x) >0$ approaching a fixed constant $\rho_\star \geq 0$ at $\pm \infty$,
\bel{eq:203bis}
(SC) \qquad \int_\RR \widetilde{D}[\rho]\, dx 
\geq C_0 
\int_\RR \left( \rho_{xx}^2 + {\rho_x^4 \over \rho^2} \right) \, {\mu(\rho)\kappa(\rho) \over \rho} \, dx.
\ee
  

\subsection{A nonlinear Sobolev inequality}

In the case that $\mu$ and $\kappa$ are power laws, the following theorem provides the key to understanding the 
strong coercivity condition proposed in the present work. (Related inequalities can be found in Lions and Villani~\cite{LV}.)  

\begin{theorem}[The strong coercivity condition for power laws] 
Consider positive functions $f: \DD \to (0, +\infty)$ defined on the torus or the real line and, more specifically: 
\begin{itemize}

\item If $\DD = \Tbb$, then consider $f$ in $H^2(\Tbb)$.

\item If $\DD = \RR$, then the functions $f$ approach a constant $\rho_\star>0$ at $\pm \infty$ and $(f-\rho_\star) \in H^2(\RR)$.
\end{itemize}
Then, the inequality 
\be \label{hummingbird0}
\int_{\DD} f^a (f_{xx})^2 \, dx \geq \left(\frac{a-1}{3}\right)^2 \int_{\DD} f^{a-2} (f_x)^4 \, dx
\ee
holds for any $a>1$, in which the constant in the right-hand side is optimal. Furthermore, there is no analogous estimate when $a=1$, in the sense that 
if the inequality 
\be
\label{hummingbird}
\int_{\DD} f (f_{xx})^2 \, dx \geq c \int_{\DD} f^{-1} (f_x)^4\, dx
\ee
holds for all functions $f$ satisfying the above requirements, then the constant $c \geq 0$ in \eqref{hummingbird} must vanish.    
\end{theorem}

Interestingly enough, our coercivity inequality enjoys many symmetries: it is invariant under the translation $f \mapsto f(x_0 + \cdot)$, 
as well as the multiplication by a constant 
$f \mapsto \lambda f$ and the dilation $f \mapsto f(\lambda \cdot)$ (this last property making sense if $\DD = \RR$, only).
This inequality also depends on our choice of boundary conditions, and the conclusion of the above theorem becomes false if, for instance, 
$f$ is taken to be any smooth function defined on the interval $[0,1]$ (without imposing periodic boundary conditions) or if
it is a function on $\RR$ admiting distinct limits at $\pm \infty$.
 
\begin{proof} It suffices to consider smooth functions $f$, since the general case follows by a straighforward density argument.

\noindent
{\bf 1. The torus with exponent $a = 1$.} In order to show that there does not exist $c>0$ such that~(\ref{hummingbird}) holds for all smooth, positive 
functions defined on the torus, we argue by contradiction and assume that it does hold for some positive $c$.

Then, identifying the torus with the interval $\left( -\frac{1}{2},\frac{1}{2} \right)$, we introduce the family of functions 
$$
\zeta_{\alpha,\eps}(x) = \left( \eps + |x|^2 \right)^{\alpha/2} \quad \mbox{if $|x| \leq 1/4$}
$$
(defined for $\eps>0$ and $\alpha>1$), which we extend to the torus so that it is smooth, bounded away from zero, and enjoys uniform bounds on its first and second derivatives 
outside of $\left( -\frac{1}{4},\frac{1}{4} \right)$. Then, we write 
$$
\int_{-1/4}^{1/4} \zeta_{\alpha,\eps}^{-1} \left( {\zeta_{\alpha,\eps}}_x \right)^4 \, dx 
= \int_{-1/4}^{1/4} \alpha^4 x^4 (\eps + x^2)^{\frac{3\alpha}{2} - 4} \, dx \overset{\eps \to 0}{\longrightarrow} \int_{-1/4}^{1/4}
\alpha^4 |x|^{3\alpha - 4}\, dx = \frac{2 \alpha^4}{3 \alpha - 3} \left( \frac{1}{4} \right)^{3\alpha - 3}, 
$$
and, on the other hand, again when $\eps \to 0$, 
$$
\int_{-1/4}^{1/4} \zeta_{\alpha,\eps}  \left( {\zeta_{\alpha,\eps}}_{xx} \right)^2 \, dx \to 
\alpha^2 (\alpha-1)^2 \int_{-1/4}^{1/4} |x|^{3\alpha-4}\, dx = \alpha^2 (\alpha-1) \frac{2}{3}  \left( \frac{1}{4} \right)^{3\alpha - 3}.
$$
If~(\ref{hummingbird}) holds, then
$$
c \int_{-1/2}^{1/2} \zeta_{\alpha,\eps}^{-1} \left( \zeta_{\alpha,\eps}^{-1} \right)^4 \, dx 
\leq \int_{-1/2}^{1/2} \zeta_{\alpha,\eps}  \left( {\zeta_{\alpha,\eps}}_{xx} \right)^2 \, dx,
$$
and letting $\eps$ go to zero and using the above gives, for some constant $A$
$$
\alpha^2 (\alpha-1) \frac{2}{3}\left( \frac{1}{4} \right)^{3\alpha - 3} \geq 
c \, \frac{2 \alpha^4}{3 \alpha - 3} \left( \frac{1}{4} \right)^{3\alpha - 3} - A.
$$
Finally, letting $\alpha$ go to $1$ leads to the desired contradiction.

\vskip.15cm 

\noindent
{\bf 2. The torus with general exponent $a \neq 1$.} Fix $a \neq -2$, and $f$ smooth and positive. Setting $h := f^{\frac{3}{a+2}}$, we then observe that the inequality
$$
\int f (f_{xx})^2 \, dx \geq c \int f^{-1} (f_x)^4\, dx
$$ 
is equivalent to
$$
\int h^a (h_{xx})^2 \, dx \geq \left(
 \left( \frac{a-1}{3} \right)^2 + c \left( \frac{a+2}{3} \right)^2 \right) \int_\Tbb h^{a-2} (h_x)^4\, dx.
$$
Therefore,~(\ref{hummingbird0}) follows from Step 1. Finally, the case $a = -2$ can be obtained by a limiting argument from the case $a \neq -2$, 
since the constant does not blow up as $a \to -2$.

\vskip.15cm 

\noindent
{\bf 3. The real line.} The invariance properties of the equation allow to deduce the results for the real line from the results for the torus, at least when $f_x$ is compactly supported. Moreover, an elementary density argument then leads to the desired conclusion.
\end{proof}


\subsection{Application to the strong coercivity condition}
\label{sec:SCC} 

Recall that $\mu$ and $\kappa$ are smooth functions mapping $(0,+\infty)$ to itself, and define
$$
\Dcalh[\rho] := \int_{\DD} \left( \frac{\mu(\rho)}{\rho} \rho_x \right)_x \left( \kappa(\rho) \rho_{xx} + \frac{1}{2} \kappa'(\rho) \rho_x^2 \right) \, dx.
$$
(this corresponds to the second term in $\int \Dt[\rho]\, dx$, the first one being non-negative if $p'(\rho) \geq 0$).
Expanding in the above formula and then integrating by parts, one sees easily that
\be
\label{eagle}
\Dcalh[\rho] = 
\int_{\DD} \left( \rho_{xx}^2 + \zeta(\rho) {\rho_x^4 \over \rho^2} \right) 
 \frac{\mu(\rho)\kappa(\rho)}{\rho} \, dx,  
\ee
with, as was defined in the introduction, 
$$
\zeta
= - {1 \over 3} \left( {1 \over 2} {\kappa'' \over \kappa}  + \left(\frac{\mu}{\rho}\right)'' {\rho \over \mu}\right).  
$$
Recall that the strong coercivity condition (SC) is satisfied if there exists a constant $C_0=C_0(\mu,\kappa)>0$ such that, for all function $\rho$ 
satisfying the 
boundary conditions specified in Section~\ref{sectionNSK},
$$
(SC) \qquad \quad \Dcalh[\rho]
\geq C_0 
\int_{\DD} \left( \rho_{xx}^2 + {\rho_x^4 \over \rho^2} \right) 
 \frac{\mu(\rho)\kappa(\rho)}{\rho} \, dx.
$$
The following theorem describes settings where the above inequality holds. Clearly, it is the case when the coefficient 
of $(\rho_x)^4$ is non-negative, but this may not be a physically realistic condition, so we consider the following broad classes of functions. 

\begin{theorem}[Sufficient conditions ensuring the strong coercivity condition] 
\label{toucan}
\begin{itemize}
\item[(i)] Fix $\mu$, $\kappa$ (positive, smooth functions on $(0,+\infty)$), and $\rho_0>0$. 
Then, there exists $\eps>0$ such that (SC) holds as soon as $|\rho(x)-\rho_0|<\eps$.
\item[(ii)] Assume $\mu(\rho) = \rho^\alpha$ and $\kappa(\rho) = \rho^\beta$. Then, the condition (SC) holds for some $C_0$ if and only if
$$
2 \alpha - 4 < \beta < 2 \alpha -1.
$$
If this condition is not satisfied, the functional $\Dcalh[\cdot]$ is not even positive.

\item[(iii)] Assume $\mu(\rho) = \rho^\alpha$ and $\kappa(\rho) = \rho^\beta$, and that the above inequality is not satisfied.
Then, there exists $c \in (0,1)$ such that the condition (SC) holds for all functions $\rho$ satisfying $|\rho(x)-\rho_0| < c \rho_0$ (for all $x$ and some $\rho_0>0$). 

\item[(iv)] If the function $\zeta$ is positive and there exists $c_0>0$ such that $\zeta(\rho) \geq c_0 \rho^{-2}$, then the condition (SC) holds. 
\end{itemize}
\end{theorem}

\begin{proof} (i) Set $\rho = \rho_0 + \eps \varphi$, where $|\varphi(x)| \leq 1$ for all $x$. The functional $\Dcalh$ becomes
$$
\Dcalh[\rho] \simeq \int_{\DD} \left( 
\eps^2 \frac{\mu(\rho) \kappa(\rho)}{\rho} (\varphi_{xx})^2 - \eps^4 \left(\frac{1}{6} \frac{\mu}{\rho}\kappa'' 
+ \frac{1}{3} \left(\frac{\mu}{\rho}\right)'' \kappa \right) (\varphi_x)^4 \right) \, dx.
$$
Taking $\eps$ sufficiently small, $\mu(\rho) \kappa(\rho) / \rho$ becomes bounded from below, and the result follows then 
if $\DD = \TT$ by the Sobolev embedding theorem, and if $\DD = \RR$ by the Gagliardo-Nirenberg inequality
$$
\| \varphi_x \|_4^2 \lesssim \| \phi \|_\infty \| \varphi_{xx} \|_2.
$$ 

\vskip.15cm 

\noindent
(ii) For $\mu(\rho) = \rho^\alpha$ and $\kappa(\rho) = \rho^\beta$, the expression~\eqref{eagle} becomes
$$
\Dcalh[\rho] = \int_{\Tbb} \left( \rho^{\alpha+\beta-1} \rho_{xx})^2 + \left( \frac{1}{6}\beta(\beta-1) + \frac{1}{3} (\alpha-1)(\alpha-2) \right) \rho^{\alpha+\beta-3}
(\rho_x)^4 \right) \, dx.
$$
Comparing this with~(\ref{hummingbird0}) gives the condition
$$
\frac{1}{6} \beta (\beta - 1) + \frac{1}{3} (\alpha-1)(\alpha-2) < \left( \frac{\alpha + \beta -2}{3} \right)^2, 
$$
which results in the range for $\beta$ in the theorem.

\vskip.15cm 

\noindent
(iii) This follows from the item (i) and a scaling argument.

\vskip.15cm 

\noindent(iv) This is immediate. 
\end{proof}

We conclude with a proposition which suggests a more general characterization of pairs $(\mu,\kappa)$ ensuring (SC).

\begin{proposition}[A characterization of the strong coercivity condition] 
When $\mu$ and $\kappa$ are power functions, say $\mu(\rho) = \rho^\alpha$ and $\kappa(\rho) = \rho^\beta$, the following three statements are equivalent: 

\begin{itemize}

\item[(i)] $\displaystyle \Dcalh[\rho] \geq 0$ for any positive and smooth function $\rho$.

\item[(ii)] $2 \alpha - 4 < \beta < 2 \alpha -1$.

\item[(iii)] $\Dcalh[\rho]$ can be written in the form 
$$
\Dcalh[\rho] = \int_{\Tbb} \frac{\mu(\rho) \kappa(\rho)}{\rho} \frac{1}{F(\rho)} \left| \big( \sqrt{F(\rho)} \rho_x \big)_x \right|^2 \, dx, 
$$
in which $F$ is a smooth and positive function.
\end{itemize}
\end{proposition}

\begin{proof} The equivalence of (i) and (ii) was already established in Theorem~\ref{toucan}. It is clear that (iii) implies (i). So we are left with proving that 
(ii) implies (iii). Comparing the expression in (iii) with~(\ref{eagle}), one sees that the functions 
$$
\displaystyle u(\rho)=\frac{\mu(\rho) \kappa(\rho)}{\rho} \frac{F'(\rho)}{F(\rho)}, 
\qquad 
\rho>0, 
$$
has to solve the ordinary differential equation in $\rho$ 
$$
\frac{1}{3} u'(\rho)- \frac{1}{4} \frac{\rho}{\mu(\rho) \kappa(\rho)} u^2(\rho) = \frac{2}{3} \left(\frac{\mu(\rho)}{\rho^2} \right)' \kappa(\rho)
+ \frac{1}{3} \left(\frac{\mu(\rho)}{\rho^2} \right)'' \rho \kappa(\rho) + \frac{\mu(\rho)}{6 \rho} \kappa''(\rho). 
$$

If $\mu$ and $\kappa$ are powers of $\rho$, this equation becomes
$$
\frac{1}{3} u'(\rho) - \frac{1}{4} \rho^{1-\alpha-\beta} u^2 
= \left(
\frac{1}{3}(\alpha-2)(\alpha-1) + \frac{1}{6}\beta(\beta-1) \right) \rho^{\alpha + \beta -3}.
$$
Try the ansatz $u = A \rho^\lambda$. First, one needs $\lambda = \alpha + \beta -2$. Next, $A$ has to solve the quadratic equation
$$
\frac{3}{4} A^2 - (\alpha + \beta -2) A + (\alpha-2)(\alpha-1) + \frac{1}{2} \beta (\beta-1) = 0.
$$
This equation has real solutions if and only if its discriminant is non-negative, that is, 
$$
\Delta = (\alpha + \beta -2)^2 - 3 \left( (\alpha -2)(\alpha-1) + \frac{1}{2} \beta (\beta-1) \right) \geq 0.
$$
Solving this inequality gives the condition (ii). If it is satisfied, we can come back to $F$, and obtain the solutions
$$
F = \rho^s,
\qquad \quad s = \frac{2}{3} \left( \alpha + \beta -2 \right).
$$
\end{proof}


\section{Finite energy solutions to the Navier--Stokes--Korteweg system}
\label{sec:3}

\subsection{Existence theory for non-cavitating solutions ($\kappa \equiv 0$ or $\geq 0$ and $\DD=\TT$ or $\RR$)}

We begin by establishing an existence theory under the non-cavitating condition (NC) introduced in~(\ref{behave}). 

\begin{theorem}[Non-cavitating finite energy solutions to the Navier-Stokes Korteweg system] 
\label{guillemot1}
Consider the initial value problem associated to the NSK system \eqref{eq:101} posed on a domain $\DD$, which is 
the torus or the real line (with $\rho_\star \geq 0$ fixed), 
and assume that 
the non-cavitating condition (NC) and the strong coercivity condition (SC) hold.
Then, given any initial data $(\rho_0,u_0)$ with finite total energy $\Escr_0$ and 
effective energy $\Escrt_0$, satisfying the non-cavitating condition $\rho_0(x) > 0$ for all $x \in \DD$, 
there exists a global-in-time weak solution $(\rho, u): [0, +\infty) \times \DD \to (0, +\infty) \times \RR$
satisfying the energy estimate $(\EDscr)$ with, furthermore, $\rho(t,x)>0$ for all $(t,x)$, and 
$$
\rho \in 
\begin{cases}
L^\infty_{t,x} \cap L^\infty_{t\operatorname{loc}} H^1_{x \operatorname{loc}},      \qquad             &\kappa\equiv 0,
 \\
L^\infty_{t,x} \cap L^\infty_{t\operatorname{loc}} H^1_{x\operatorname{loc}} \cap L^2_{t\operatorname{loc}} H^2_{x \operatorname{loc}},         & \kappa>0. 
\end{cases} 
\qquad \qquad 
\qquad 
u \in L^\infty_{t\operatorname{loc}} L^2_{x\operatorname{loc}} \cap L^2_{t\operatorname{loc}} H^1_{x\operatorname{loc}},
$$ 
where $H^1, H^2$ denote the standard Sobolev spaces. 
\end{theorem}

\begin{proof} \noindent{\bf 1. Reduction to the pre-compactness of solutions to (NSK).} We need to establish that 
the pre-compactness of 
the family of NSK solutions satisfying (at the initial time) a uniform energy bound and a lower bound on the mass density. In other words,
 from any sequence satisfying these uniform bounds, we can extract a subsequence that converges (in a suitable 
topology) to a solution to  \eqref{eq:101}. This is the property required in order to deduce the strong convergence of approximate 
solutions to \eqref{eq:101} and, eventually, establish the existence of actual solutions. On the other hand, several standard methods 
are available for the construction of approximate solutions; one can for instance use Galerkin-type schemes and we refer to Gamba, J\"ungel, and Vasseur~\cite{GambaJungelVasseur} for further details.

\vskip.15cm 

\noindent{\bf 2. Energy bounds.} Motivated by the observation above, we now consider solutions $(\rho^n,u^n)$ associated with some
initial data $(\rho_0^n,u_0^n)$ whose physical energy $\Escr_0^n$ and effective energy $\Escrt_0^n$ are uniformly bounded, that is, 
$$
\sup_n \Big( \Escr_0^n + \Escrt_0^n \Big) < + \infty. 
$$
In view of the previous section and our assumption (SC), in particular, this implies the uniform bound $(\EDscr)$ for $(\rho^n,u^n)$. 

\vskip.15cm 

\noindent{\bf 3. Lower bound on the mass density.}  Here and in the following step, we rely on arguments first used in Hoff~\cite{Hoff} and Mellet and Vasseur~\cite{MelletVasseur} in order to cope with 
the viscosity term in Navier-Stokes equations. Starting with the case of the torus $\DD = \TT$, we observe that the energy provides us with the bound 
$$
\kappa(\rho^n) (\rho^n_x)^2 \in L^\infty_t L^1_x,
$$ 
while the energy and the effective energy (taken together) control 
$$
\frac{(\mu(\rho^n))^2}{(\rho^n)^3} (\rho^n_x)^2 \in L^\infty_t L^1_x.
$$ 
Thanks to the no-cavitating condition~(NC) (stated in \eqref{behave}), the latter two estimates imply the existence of $\alpha>0$ such that 
$((\rho^n)^{-\alpha})_x \in L^\infty_t L^2_x$. On the other hand, the mass $\Mscr_0^n := \int_\Tbb\rho^n\, dx$ is conserved in time
and is uniformy bounded thanks to our uniform energy bound, so that
the following lower bound holds: 
$$
\sup_{x \in \Tbb} \rho(t,x) \geq {\Mscr_0^n \over |\Tbb|}. 
$$
In turn,  by Sobolev's embedding theorem, we arrive at the lower bound $\inf_{t,x} \rho^n \geq \rho_ {\min}>0$ for all $n$.

In the case of the real line $\DD = \RR$ with $\rho_\star>0$, the same arguments yield again $((\rho^n)^{-\alpha})_x \in L^\infty_t L^2_x$ and we then 
observe 
that the bound on $\Escr$ implies the existence of constants $R,r>0$ such that, for any $x_0$, there exists $x \in [x_0-R,x_0+R]$ such that $\rho^n(t,x)>r$. We can then conclude again by Sobolev's embedding theorem.

In the case of the real line $\DD = \RR$ with $\rho_\star=0$, the same approach can once again be followed. First one shows, for any $t$, the existence of reals $R(t)$ and $r(t)$, depending only on the data, such that 
$\rho(x_0,t)>r$ for some $|x_0|<R$. Then one applies the fact that $((\rho^n)^{-\alpha})_x \in L^\infty_t L^2_x$ to get a lower bound (which goes to zero as $x \rightarrow \infty$. We skip the details. One obtains,
for any $T$, the existence of a positive function $\zeta_T$ such that $\rho^n(x,t) > \zeta_T(x,t)$ for $t<T$.

\vskip.15cm 

\noindent{\bf 4. Upper bound on the mass density.}
In the case of the torus $\DD = \TT$, observe as above that the energy and the effective energy (taken together) control 
$\frac{(\mu(\rho^n))^2}{(\rho^n)^3} (\rho^n_x)^2$ in $L^\infty_t L^1_x$.
Thanks to the mild assumption~(\ref{hypmukappa}), this implies that $\chi(\rho^n) \sqrt{\frac{e(\rho^n)}{\rho^n}}$ (where $\chi$ cuts off smoothly to $\rho^n \geq 1$)
belongs to $L^\infty_t W_x^{1,1}$. Indeed, by the Cauchy-Schwarz inequality 
and the growth condition on the pressure, we have
$$
\aligned
\int_\Tbb \chi(\rho^n) \Big| \partial_x \sqrt{e(\rho^n) / \rho^n} \Big| \, dx
& \lesssim 
\int_\Tbb \chi(\rho^n)\left| \frac{\sqrt{e(\rho^n)}}{(\rho^n)^{3/2}} \rho^n_x \right| \, dx
\\
& \leq 
\left( \int_\Tbb \chi(\rho^n)e(\rho^n)\, dx \int_\Tbb \chi(\rho^n)\frac{(\rho^n_x)^2}{(\rho^n)^3}\, dx \right)^{1/2} \leq C(\mathscr{E}_0,\widetilde{\mathscr{E}_0}).
\endaligned
$$
On the other hand, the mass equation gives
$$
\inf_{x \in \Tbb} \rho^n(t,x) \leq {\Mscr_0^n \over |\Tbb|}.  
$$
These two observations imply the existence of the uniform upper bound $\sup_{t,x} \rho^n \leq \rho_{\max}$ uniformly in $n$.

In the case of the real line $\DD = \RR$, the second observation above needs to be modified, as follows: by virtue of the boundedness of the renormalized energy, there exists a constant $ \rho_{\max}>0$ such that, for any $x_0$, there exists $x \in [x_0-1,x_0+1]$ such that $\rho^n(t,x) \leq  \rho_{\max}$.

\vskip.15cm 

\noindent{\bf 5. Uniform estimates in Sobolev spaces.} The uniform estimate $(\EDscr)$ and the upper and lower bounds on $\rho^n$ imply (with uniform bounds)
$$
\rho^n \in L^\infty_{t \operatorname{loc}} H^1_{x \operatorname{loc}} \cap L^2_{t\operatorname{loc}} H^2_{x\operatorname{loc}},
\qquad
\quad u^n \in L^\infty_{t\operatorname{loc}} L^2_{x\operatorname{loc}} \cap L^2_{t\operatorname{loc}} H^1_{x\operatorname{loc}}.
$$
Furthermore, by returning to the system (NSK) satisfied by $(\rho^n,u^n)$, we also deduce uniform bounds for the time derivatives 
$$
\rho^n_t \in L^\infty_{t\operatorname{loc}} H^{-1}_{x\operatorname{loc}},
\qquad \quad
u^n_t \in L^2_{t\operatorname{loc}} W^{-1,1}_{x\operatorname{loc}}.
$$

\vskip.15cm 

\noindent{\bf 6. Passage to the limit.} Using standard Sobolev compactness theorems and, in particular, Aubin-Lions lemma, we deduce that there exists a limit $(\rho,u)$ such that, 
for all $\eps \in (0,1)$, 
\begin{equation*}
\begin{split}
& \rho^n \to \rho \quad 
\begin{cases}
& \mbox{weakly--$\star$ in $L^\infty_{\loc, t,x}$},
\\
& \mbox{weakly in $L_t^\infty H_x^1$ and strongly in $L^p_t H^{1-\eps}_x$ locally for all $p \in [1, +\infty)$},
\\
& \mbox{weakly in $L_t^2 H_x^2$ and strongly in $L^2_t H^{2-\eps}_x$ locally if $\kappa>0$}, 
\end{cases}
\\
& u^n \to u \quad
\begin{cases}
& \mbox{weakly-$\star$ in $L_t^\infty L_x^2$ locally for all $p<+\infty$,} 
\\
&   \mbox{weakly in $L_t^2 H_x^1$ locally, and strongly in $L^2_t H^{1-\eps}_x$ locally}.
\end{cases}
\end{split}
\end{equation*}
It is immediate to now pass to the limit in the weak formulation of the system (NSK) and check that $(\rho,u)$ solves the equations in a weak sense.

Observe that, in the case that the capillarity $\kappa$ vanishes identically, then the convergence property 
$\rho^n \to \rho$ (weakly in $L_t^2 H_x^2$) is ``lost'' but, simultaneously, the third-order terms in our system are gone, so that the convergence properties above are strong enough to allow us to conclude. 
\end{proof}
 

\subsection{Existence theory for cavitating solutions ($\kappa=0$ and $\DD = \TT$)}

We are now interested in cavitating solutions, which we will obtain as limits of the non-cavitating solutions in Theorem~\ref{guillemot1}.  
We begin by considering the Navier-Stokes (NS) system
\bel{eq:101-NS} 
\aligned
\rho_t + (\rho u)_x =& \, 0, 
\\
(\rho u)_t + \big( \rho u^2 + p(\rho) \big)_x =& \, (\mu(\rho) \, u_x)_x, 
\endaligned
\ee 
and we focus on the Cauchy problem posed on the torus (while the following two subsections concern the real line).

\begin{theorem}[Cavitating finite energy solutions to the Navier-Stokes system] 
\label{guillemot2}
Consider the system \eqref{eq:101-NS}  
posed on the torus with prescribed initial data $(\rho_0,u_0)$ with finite energy $\Escr_0$ and finite effective energy $\Escrt_0$. Assume that the data satisfy 
\bel{eq:add-data}
\rho_0 \, |u_0|^{2+s} \in L^1(\TT)
\ee
for some $s>0$ and that 
\bel{pinson}
\qquad {p(\rho)^2 \over \mu(\rho) \rho^{s/2}} \lesssim 1 \qquad \mbox{for small $\rho$.} 
\ee
Then, the  Cauchy problem associated with \eqref{eq:101-NS} 
admits a global-in-time weak solution  $(\rho, u): [0, +\infty) \times \TT \to [0, +\infty) \times \RR$,
which may contains vacuum regions (and is understood in the sense of Definition~\ref{def:weakcav}, below),  
and has finite energy and dissipation in the sense $(\EDscr)$ 
with, furthermore, $\rho \in L^\infty_{t,x}$ and $\mu(\rho) \in \mathcal{C}^{1/4}_{t,x}$.
\end{theorem}
 
We emphasize that the mass density $\rho$ is continuous, so that the following notion of weak solution applies. 

\begin{definition} 
\label{def:weakcav}
 Under the assumptions and regularity properties in Theorem~\ref{guillemot2}, 
a {\bf weak solution with cavitation} to the system \eqref{eq:101-NS} is defined   as follows:
the uniform estimate $(\EDscr)$ controls $e(\rho)$ in $L^\infty_t L^1_x$, $\rho u^2$ in $L^\infty_t L^1_x$, and
$\mu(\rho) u_x^2$ in $L^2_t L^2_x$, so 
that the conservative variable and flux terms $\rho$, $\rho u$, $\rho u^2$, $p(\rho)$, and $\mu(\rho) u_x^2$ are locally integrable functions. 
For instance, for $\mu(\rho) u_x$ and by
Cauchy-Schwarz, $\mu(\rho) u_x \lesssim \mu(\rho) u_x^2 + \mu(\rho)$. 
\end{definition}

Before we proceed with the proof of Theorem~\ref{guillemot2}, 
we establish a preliminary lemma, which takes advantage of the viscosity term in order to derive a better integrability property for the velocity. 
Observe that the alternative approach in Theorem~\ref{guillemot3}, below,
relies on a hyperbolic-type property of higher-integrability, which eventually 
allows us to remove the additional integrability condition \eqref{eq:add-data}. 
 
\begin{lemma}[Higher-order integrability of the velocity]
\label{loriot}
Under the conditions in Theorem~\ref{guillemot2}, for each $T>0$ there exists a constant $C(T,\Escr_0,\Escrt_0)>0$ 
such that any non-cavitating solution of the Navier-Stokes system \eqref{eq:101-NS} defined 
on the torus $\TT$ satisfies for all $t<T$
$$
\int_\TT (\rho |u|^{2+s})(t, \cdot) \, dx + \int_0^t \int_\TT |u|^s \mu(\rho) |u_x|^2 \,dtdx
\lesssim \int_\TT \rho_0 |u_0|^{2+s} \, dx + C(T,\Escr_0,\Escrt_0).
$$
\end{lemma}

\begin{proof} We follow an argument in Mellet and Vasseur~\cite{MelletVasseur} and multiply the momentum equation in \eqref{eq:101-NS} 
by $|u|^s u$. After integrating in space, we obtain 
$$
\int_\TT \left( \rho u_t + \rho u u_x\right) |u|^s u \, dx + \int_\TT p(\rho)_x |u|^s u \, dx 
- \int_\TT (\mu(\rho) u_x)_x |u|^s u \, dx = 0.
$$
The first and third terms above can be rewritten in the form 
\bel{eq:4001}
\frac{1}{2+s} \frac{d}{dt} \int_\TT \rho |u|^{2+s}\, dx + (s+1) \int_\TT \mu(\rho) u_x^2 |u|^s \, dx,
\ee
while, by integrating by parts and using Cauchy-Schwarz inequality,  the second term can be bounded as follows: 
$$
\left| \int_\TT p(\rho)_x |u|^s u \, dx \right| \lesssim \int_\TT p(\rho) |u|^s |u_x| \, dx 
\leq \eps \int_\TT \mu(\rho) u_x^2 |u|^s \, dx + \frac{1}{\eps} \int_\TT \frac{p(\rho)^2}{\mu(\rho)} |u|^s \, dx.
$$
The first term in the right-hand side above is controlled by the dissipation term in \eqref{eq:4001}, provided $\eps$ is sufficiently small. 
For the second term, we write the integrand in the form $\rho^{s/2} u^s \frac{p(\rho)^2}{\mu(\rho)\rho^{s/2}}$
and we observe that:
\begin{itemize}

\item[(i)] The factor $\rho^{s/2} u^s$ is bounded in $L^\infty_t L^{2/s}_x$ (in terms of the physical energy) and thus in 
$L^\infty_t L^1_x$. 

\item[(ii)] The factor $\frac{p(\rho)^2}{\mu(\rho)\rho^{s/2}}$ belongs to $L^\infty_{t,x}$
since (\ref{pinson}) is assumed and the mass density is uniformly bounded in $L^\infty_{t,x}$ (cf.~the proof of Theorem~\ref{guillemot1} or see 
\eqref{eq:4002}, below).
\end{itemize}
\end{proof}

\begin{proof}[Proof of Theorem~\ref{guillemot2}]
 {\bf 1. Approximation from the non-cavitating regime.} 
Consider a viscosity coefficient $\mu=\mu(\rho)$ and initial data $(\rho_0,u_0)$ be as in the statement of the theorem,
and let us solve the NS system with non-cavitating data $(\rho_0+1/n, u_0)$ and viscosity $\mu^n(\rho) = \mu(\rho)+1/n$. 
By Theorem~\ref{guillemot1}, there exists a solution $(\rho^n,u^n)$ which enjoys 
uniform physical energy and effective energy bounds and such that
\bel{eq:4002}
\inf_{t,x} \rho^n(t,x) \geq \rho_{\min}^n>0,
\quad 
\qquad
\sup_{t,x} \rho^n(t,x) < \rho_{\max}, 
\ee
in which $\rho_{\min}^n>0$ may depend on $n$ but $\rho_{\max}$ is independent of $n$. 
We are going to establish that these solutions converge (as $n \to \infty$) to a solution to the NS system, denoted by $(\rho,u)$. 

\vskip.15cm 

\noindent{\bf 2. Uniform bounds on the mass density and its derivatives.} 
Thanks to the energy bound, $\rho_n$ is bounded in $L^\infty_t L^\gamma_x$ locally, 
so that, upon selecting a subsequence if necessary, we assume ($n \to +\infty$)
$$
\rho^n \to\rho \quad \mbox{weakly-}\star \text{ in } L^\infty_t L_x^\gamma. 
$$
On one hand, the inequality 
$$
| (\mu^n(\rho^n))_x |
 \lesssim
\frac{\mu^n(\rho^n)}{\rho^n}  \, | \rho^n_x | 
\lesssim 
\rho_{\max}^{1/2} \frac{\mu^n(\rho^n)}{(\rho^n)^{3/2}} \, | \rho^n_x| 
$$
implies that $\| \mu^n(\rho^n)_x \|_{L^\infty_t L^2_x} \lesssim C(\Escr_0,\Escrt_0,\Mscr_0)$, 
whereas, on the other hand, 
$$
| (\mu^n(\rho^n))_t | 
\lesssim \mu^n(\rho^n) |u^n_x| + \frac{\mu^n(\rho^n)}{(\rho^n)^{3/2}} |\rho^n_x| \sqrt{\rho^n} \, |u^n|,
$$
implies that $\| \mu^n(\rho^n)_t \|_{L^2_t L^2_x + L^\infty_t L^1_x} \lesssim C(\Escr_0,\Escrt_0,\Mscr_0)$.

From the above estimates, we now deduce that $\mu^n(\rho^n)$ is bounded in $\mathcal{C}^{1/4}$, indeed: 
$$
\aligned
\left| \mu^n(\rho^n)(t,x) - \mu^n(\rho^n)(s,y) \right| 
\leq 
& \left| \mu^n(\rho^n)(t,x) - \mu^n(\rho^n)(t,y) \right| + \left| \mu^n(\rho^n)(t,y) - \frac{1}{2h} \int_{[y-h;y+h]} \mu^n(\rho^n(t,\cdot)) \right| 
\\
& + \left| \frac{1}{2h} \int_{[y-h;y+h]} \mu^n(\rho^n(t,\cdot))- \frac{1}{2h} \int_{[y-h;y+h]} \mu^n(\rho^n(s,\cdot)) \right|
\\ 
& + \left| \mu^n(\rho^n)(s,y) - \frac{1}{2h} \int_{[y-h;y+h]} \mu^n(\rho^n(s,\cdot)) \right| 
\\
\lesssim 
& \sqrt{|x-y|} + \sqrt{h} + \frac{\sqrt{|t-s|}}{\sqrt{h}} + \frac{|t-s|}{h} \lesssim \sqrt{|x-y|} + |t-s|^{1/4} + \sqrt{|t-s|}.
\endaligned
$$ 
upon choosing $h = \sqrt{|t-s|}$.
Finally, thanks to Arzela-Ascoli theorem and after taking a subsequence if necessary, we have the strong convergence property 
$$
\mu(\rho^n) \to \mu(\rho) \quad \text{ in } \mathcal{C}^{1/4}_\loc.
$$ 

\vskip.15cm 


\noindent{\bf 3. Uniform bounds on the pseudo-momentum} $m_\sharp^n := \mu_\sharp^n(\rho^n) u^n$, 
with $\mu_\sharp^n(\rho) := \min (\mu^n(\rho),\rho)$. Since the mass density is uniformly bounded, it is clear that
\bel{chardonneret1}
\| m_\sharp^n \|_{L^\infty_t L^2_x} \lesssim \Escr_0, 
\ee
while its $x$-derivative can be estimated as follows:
\bel{chardonneret2}
| (m_\sharp^n)_x | 
\lesssim \mu_\sharp^n(\rho^n) |u^n_x| 
        + \frac{\mu_\sharp^n(\rho^n)}{\rho^n} \, |\rho^n_x| \, | u^n | 
\leq 
\mu^n (\rho^n) |u^n_x | +  \frac{\mu^n (\rho^n)}{(\rho^n)^{3/2}} \, |\rho^n_x| \sqrt{\rho_n} \, |u^n|,  
\ee 
implying $\|(m_\sharp^n)_x \|_{L^2_t L^1_x} \leq C(\Escr_0,\Escrt_0,\Mscr_0)$. 
To estimate the time derivative, we use the NS equations to replace the time derivatives of $\rho^n, u^n$, as follows:
$$
\aligned
(m_\sharp^n)_t  
= \, & \mu_\sharp^n(\rho_n) u^n_t + (\mu_\sharp^n)'(\rho^n) \rho^n_t u^n 
\\
= \, & - \mu_\sharp^n(\rho^n) u^n u^n_x - \frac{\mu_\sharp^n(\rho^n)}{\rho^n} (p(\rho^n))_x 
       + \frac{\mu_\sharp^n(\rho^n)}{\rho^n} (\mu^n(\rho^n) u^n_x)_x - (\mu_\sharp^n)'(\rho^n) \rho^n u^n_x u^n
       - (\mu_\sharp^n)'(\rho^n) \rho^n_x u^n u^n  
\\
= \, & \Big( - (\Gamma(\rho^n))_x + (\frac{\mu_\sharp^n(\rho^n)}{\rho^n} \mu^n(\rho^n) u^n_x)_x - (\mu_\sharp^n(\rho^n)(u^n)^2)_x \Big)
\\
& + \Big(  - (\frac{\mu_\sharp^n}{\rho^n})'(\rho^n) \rho^n_x \mu^n(\rho^n) u^n_x
 + (\mu_\sharp^n)'(\rho^n) \rho^n u^n_x u^n- \mu_\sharp^n(\rho^n) u^n u^n_x \Big) =: \NI + \NII,
\endaligned
$$
where the function $\Gamma=\Gamma(\rho)$ satisfies $\Gamma(0)=0$ and $\Gamma' = \frac{\mu_\sharp(\rho) p'(\rho)}{\rho}$. It is now easy to see that
$\left\| \NI \right\|_{L^2_{t \operatorname{loc}} H^{-1}} \leq C (\Escr_0,\Escrt_0,M)$ while $ \left\| \NII \right\|_{L^2_{t \operatorname{loc}} L^1_x}  \leq C (\Escr_0,\Escrt_0,\Mscr_0)$.
Therefore, by Sobolev embedding, we find 
\bel{chardonneret3}
\| (m_\sharp^n)_t \|_{L^2_{t \operatorname{loc}} H^{-1}} \leq C (\Escr_0,\Escrt_0,\Mscr_0).
\ee

\vskip.15cm 


\noindent{\bf 4. Convergence of the fluid velocity.} The Aubin-Lions lemma, combined with the estimates~(\ref{chardonneret1}),~(\ref{chardonneret2}) and (\ref{chardonneret3}), gives 
the existence of $m_\sharp$ that a subsequence of $m_\sharp^n$ converges in $L^2_t L^2_x$ locally to $m_\sharp$. 
Upon taking a further subsequence, we can assume that it also converges almost everywhere. We then set
$$
u = \begin{cases}
\frac{m_\sharp}{\mu_\sharp(\rho)}, &  \rho \neq 0,
\\ 
0,                  & \rho = 0. 
\end{cases}
$$ 
With this definition, $u^n$ converges to $u$ almost everywhere on the set $\big\{ \rho > 0 \big\}$.
Next, choose $\eps>0$. Recall that $\mu(\rho^n)$ converges to $\mu(\rho)$ in $\mathcal{C}^{1/4}$. In particular, for 
all sufficiently large $n$, we have $\mu(\rho^n) > \frac{\eps}{2}$ on the set $O_\eps = \big\{\mu(\rho) > \eps \big\}$, which
 gives us a uniform bound for $u^n_x$ in $L^2_{t,x}(O_\eps)$.
Taking a further subsequence if necessary, we can assume that $u^n_x$ converges weakly to some $v$ in $L^2_{t,x}(O_\eps)$. It is easy to check that
$v = u_x$. 
By a diagonal argument, we can further achieve that, for each $\eps$, $u^n_x$ converges weakly to $u_x$ in $L^2_{t,x}(O_\eps)$.

\vskip.15cm 


\noindent
{\bf 5. Passage to the limit in the momentum.} Denoting by $\chi$ the characteristic function of the interval $[-1,1]$
and for arbitrary $L, T>0$, we have 
$$
\aligned
\int_0^T \int_\TT | \rho^n u^n - \rho u | \, dtdx 
\leq 
& \int_0^T \int_\TT \left| \rho^n u^n \chi \left( \frac{u^n}{L} \right) - \rho u \chi \left( \frac{u}{L} \right) \right| \, dtdx 
\\
& + \int_0^T \int_\TT  \rho^n | u^n | \,\Bigg( 1- \chi \left( \frac{u^n}{L} \right) \Bigg) \, dtdx 
   + \int_0^T \int_\TT \rho |u|  \, \Bigg( 1-\chi \left( \frac{u}{L} \right) \Bigg) \, dtdx.
\endaligned
$$
The first term in the above right-hand side converges to $0$ as $n\to \infty$ (by the dominated convergence theorem), since we have uniform bounds on the integrands. For the second and third terms, we rely on the uniform bound on the physical energy and write 
$$
\limsup_{n \to +\infty} \int_0^T \int_\TT | \rho^n u^n - \rho u | \, dtdx 
\leq \frac{C(\Escr_0)T}{L}.
$$
Letting $L$ tend to infinity, we see that $\rho^n u^n$ converges to $\rho u$ in $L^1_{t,x}$.

\vskip.15cm 


\noindent{\bf 6. Passage to the limit in the momentum flux.} The momentum flux is treated in a similar way to the previous step, by now writing 
$$
\aligned
\int_0^T \int_\TT \left| \rho^n (u^n)^2 - \rho u^2 \right| \, dtdx 
\leq 
& \int_0^T \int_\TT \left| \rho^n (u^n)^2 \chi \left( \frac{u^n}{L} \right) - \rho u^2 \chi \left( \frac{u}{L} \right) \right| \, dtdx 
\\
& + \int_0^T \int_\TT \rho^n (u^n)^2 \, \Bigg( 1- \chi \left( \frac{u^n}{L} \right)\Bigg) \, dtdx 
+ \int_0^T \int_\TT  \rho u^2 \, \Bigg(1- \chi \left( \frac{u}{L} \right) \Bigg) \, dtdx. 
\endaligned
$$
At this juncture, we rely on the higher-integrability property established in Lemma~\ref{loriot} and obtain
$$
\limsup_{n \to +\infty} \int_0^T \int_\DD \left| \rho^n (u^n)^2 - \rho u^2 \right| \, dtdx \leq \frac{C(\Escr_0)T}{L^s}.
$$
By letting $L$ tend to infinity, we conclude that $\rho^n (u^n)^2$ converges to $\rho u^2$ in $L^1_{t,x}$ locally.

\vskip.15cm 


\noindent{\bf 7. Passage to the limit in the viscous force.} This final term is more delicate if one wants to cover the cavitating regime. 
We will prove that $\mu^n(\rho^n) u^n_x \to \mu(\rho) u_x$ in the sense of distributions and that all the terms involved can be defined. Pick up any smooth function $\chi: \RR \to [0, 1]$ that vanishes identically on the interval $[-1,1]$, and is identically
equal to $1$ on $[-2,2]^c$. Then, if $\phi$ is any test-function and $\eps>0$, we have 
$$
\aligned
\iint \phi \mu^n(\rho^n) u^n_x \, dtdx 
& = \iint \phi \mu^n(\rho^n) u^n_x \chi\left( \frac{\mu^n(\rho^n)}{\eps} \right) \, dtdx
    +  \iint \phi \mu^n(\rho^n) u^n_x \left( 1 - \chi\left( \frac{\mu^n(\rho^n)}{\eps} \right) \right) \, dtdx
\\
& = \NI + \NII.
\endaligned
$$
Since $\mu^n(\rho^n)$ converges to $\mu(\rho)$ in $\mathcal{C}^{1/4}$; 
 for all  sufficiently large $n$, the functions $\chi\left( \frac{\mu^n(\rho^n)}{\eps} \right)$ are supported
on $O_{\eps/2}$. Using that $u^n_x$ converges to $u_x$ in $L^2_{t,x}(O_\eps)$ whereas $\chi\left( \frac{\mu^n(\rho^n)}{\eps} \right)\mu^n(\rho^n)$
converges to $\chi\left( \frac{\mu(\rho)}{\eps} \right)\mu^n(\rho^n)$ in $\mathcal{C}^{1/4}$, we have 
$$
\NI \to\iint \phi \mu(\rho) u_x \chi\left( \frac{\mu(\rho)}{\eps} \right) \, dtdx \qquad \text{ when } n \to +\infty.
$$
On the other hand, we have 
$$
| \NII | \lesssim \left\| \mu^n(\rho^n) u^n_x \left( 1 - \chi\left( \frac{\mu^n(\rho^n)}{\eps} \right) \right) \right\|_{L^2_{t,x}}
\leq \sup_{[0,2\eps]} \sqrt{\mu^n} \left\| \sqrt{\mu^n(\rho^n)} u^n_x \right\|_{L^2_{t,x}} \leq \sup_{[0,2\eps]} \sqrt{\mu^n} \Escr_0 
\overset{\eps \to 0, n \to \infty}{\longrightarrow} 0.
$$
Combining the estimates for $\NI$ and $\NII$ yields
$$
\iint \phi \mu^n(\rho^n) u^n_x \, dtdx \to\iint \phi \mu(\rho) u_x \, dtdx.
$$

\vskip.15cm 

\noindent
{\bf 8. Final conclusion} We have established that, as $n \to +\infty$ and in the sense of distributions,
$\rho^n \to \rho$, $\rho^n u^n \to \rho u$, $\rho^n (u^n)^2 \to \rho u^2$ and $\mu^n(\rho^n) u^n_x \to \mu(\rho) u_x$.
Furthermore, by using the upper bound on $\rho^n$, it also follows that $p(\rho^n) \to p(\rho)$. Therefore, we can pass to the limit 
in all the terms involved in the Navier-Stokes equations and we conclude that the limit $(\rho,u)$ is indeed a weak solution. 
\end{proof}


\subsection{Existence theory for cavitating solutions ($\kappa=0$ and $\DD=\RR$)}

We now generalize our analysis to the real line. 

\begin{theorem}[Cavitating finite energy solutions to the Navier-Stokes system] 
\label{guillemot2bis}
Consider the Navier-Stokes system \eqref{eq:101-NS} posed on $\RR$ with prescribed initial data $(\rho_0,u_0)$ with finite energy $\Escr_0$,
finite effective energy $\Escrt_0$, and finite mass $\mathscr{M}_0$. Assume that, for some $a>0$,
$$
\rho^a \gtrsim \mu(\rho) \gtrsim \rho \qquad \text{for $\rho$ small.}
$$
Then, there exists a global-in-time solution  $(\rho, u): [0, +\infty) \times \TT \to [0, +\infty) \times \RR$ to \eqref{eq:101-NS}, 
which possibly contains vacuum regions and  
such that $(\EDscr)$ holds with, furthermore, $\rho \in L^\infty_{t,x}$ and $\mu(\rho) \in \mathcal{C}^{1/4}_{t,x}$.
\end{theorem}

The weak solutions with cavitation above are understood in the sense of distributions, along the lines of Definition~\ref{def:weakcav}, above. 

\begin{proof} Like Theorem~\ref{guillemot2}, this theorem is established by approximation from the non-cavitating case treated in Theorem~\ref{guillemot2bis}.
We will not repeat herre all the steps already described in the proof of Theorem~\ref{guillemot2}, and we only emphasize here the novel argument required in the present proof, that is, 
the higher integrability property for the velocity. The following argument relies on a notation which will be introduced 
at the beginning of Section~\ref{sec:43}, below. 

\vskip.15cm 

\noindent
{\bf 1. A new set of entropies.} Given a parameter $a \in [0,1]$, we consider the two entropy--entropy flux pairs 
obtained by choosing the function $\psi$ to be $|v|^{a+1}$ and $v|v|^a$ in \eqref{eq:4200}, that is,  
$$
\eta^a := \eta^{|v|^{a+1}}, \qquad q^a := q^{|v|^{a+1}}, \qquad \etat^a := \eta^{v|v|^a}, \qquad \widetilde{q}^a := q^{v|v|^a}.
$$
Using \eqref{eq:chidef}--\eqref{eq:4200}, these entropies can be checked to satisfy the following pointwise bounds:
\bel{lemmaboundbis}
\aligned
& |\eta^a| \lesssim \rho^{(a+1)\theta + 1} + \rho |u|^{a+1},
\qquad \qquad
&& \left| \etat^a \right| \lesssim \rho^{(a+1)\theta+1} + \rho |u|^{a+1},
\\
& \left| q^a \right| \lesssim \rho^{\gamma + a \theta} + \rho |u|^{2+a},
&& \widetilde{q}^a \gtrsim \rho^{\gamma + a \theta} + \rho |u|^{2+a},
\\
& \left| \eta^a_m \right| \lesssim \rho^{a\theta} + |u|^a,
&& \left| \etat^a_m \right| \lesssim \rho^{a\theta} + |u|^a,
\\
& \eta_{mu}^a \gtrsim \rho^{(a-1)\theta} \langle u \rho^{-\theta} \rangle^{a-1},
&& \left| \etat^a_{m u} \right| \lesssim  \rho^{(a-1)\theta} \langle u \rho^{-\theta} \rangle^{a-1},
\\
& \left| \eta^a_{m \rho} \right| \lesssim \rho^{a \theta - 1}, 
&& \left| \etat^a_{m \rho} \right| \lesssim \rho^{a\theta - 1}.
\\
\endaligned
\ee
In particular, we observe that $\widetilde{q}^a$ is positive. 

\vskip.15cm 

\noindent
{\bf 2. A bound via the entropy dissipation of $\eta^a$.} 
We will prove that under the assumptions of the theorem, a solution of Navier-Stokes satisfies
$$
\int_0^T \int \rho^{(a-1)\theta} \langle \frac{u}{\rho^\theta} \rangle^{a-1} \mu(\rho) u_x^2 \, dtdx \lesssim C(T,\mathscr{M}_0,\Escr_0,\Escrt_0)
$$
Start by writing the conservation law for the entropy $\eta^a$:
$$
(\eta^a)_t + (q^a)_x = \eta^a_m (\mu(\rho) u_x)_x.
$$
Integrating it in time on $[0,T]$ and in space over $\RR$ gives
\begin{equation*}
\begin{split}
\int \eta^a(x,T) \, dx - \int \eta^a(x,0) \, dx & = \int_0^T \int \eta^a_m (\mu(\rho) u_x)_x\, dx,ds \\
& = - \int_0^T \int \eta^a_{mu} \mu(\rho) u_x^2\, dtdx - \int_0^T \int \eta^a_{m\rho} \rho_x \mu(\rho) u_x\, dtdx,
\end{split}
\end{equation*}
where the second equality follows by integration by parts. Using the bounds~(\ref{lemmaboundbis}) leads to
\begin{equation*}
\begin{split}
\int_0^T \int \rho^{(a-1)\theta} \langle \frac{u}{\rho^\theta} \rangle^{a-1} \mu(\rho) u_x^2 \, dtdx & \lesssim
\int \left( \rho^{(a+1)\theta + 1} + \rho |u|^{a+1}\right) (T) \, dx + \int \left( \rho^{(a+1)\theta + 1} + \rho |u|^{a+1}\right) (0) \, dx \\
& \qquad \qquad + \int_0^T \int \rho^{a \theta - 1} \mu(\rho) \left| \rho_x u_x \right| \, dx\,dt.
\end{split}
\end{equation*}
The two terms in the above right-hand side can be bounded by using the observation that
that $\rho^{(a+1)\theta + 1} + \rho |u|^{a+1} \lesssim \rho + \rho u^2 + \rho^\gamma$; for the third one, we use Cauchy-Schwarz inequality. This gives
\begin{equation*}
\begin{split}
\int_0^T \int \rho^{(a-1)\theta} \langle \frac{u}{\rho^\theta} \rangle^{a-1} \mu(\rho) u_x^2 \, dx\,dt & \lesssim \Escr_0 + \mathscr{M}_0
+ \int_0^T \int \mu(\rho) u_x^2 \, dtdx + \int_0^T \int \frac{\mu(\rho)}{\rho^2} \rho^{2a \theta} \rho_x^2\, dx\,dt.
\end{split}
\end{equation*}
The last term in the above right-hand side remains to be controlled. Since $\mu(\rho) \gtrsim \rho$, it is controlled by $\int \int \frac{\mu(\rho)^2}{\rho^3}(\rho_x)^2\, dx\,dt$.
We therefore get
$$
\int_0^T \int \rho^{(a-1)\theta} \langle \frac{u}{\rho^\theta} \rangle^{a-1} \mu(\rho) u_x^2 \, dtdx  \lesssim \Escr_0 + \mathscr{M}_0+ \Escrt_0 \langle T \rangle,
$$
which is the desired result.

\vskip.15cm

\noindent
{\bf 3. A bound via the entropy flux of $\widetilde{\eta}^a$.} 
We prove here the uniform bound
$$
\int_0^T \int_K \left( \rho |u|^{2+a} + \rho^{\gamma+a\theta} \right)\, dtdx \lesssim C(K,M_0,\Escr_0,\Escrt_0,T).
$$
which is the desired higher integrability for the velocity.
The proof is actually parallel to that of Proposition~\ref{grebe}; the difference being that, here, $\epsilon=0$, which greatly simplifies the estimates.
Thus we can skip the details and refer to the proof of Proposition~\ref{grebe}.
Begin by writing the conservation law for the entropy $\etat^a$
\begin{equation*}
\underbrace{\partial_t \etat^a(\rho,u)}_{\NI} + \underbrace{\partial_x \widetilde{q}^a(\rho,u)}_{\NII} = 
\underbrace{\etat^a_m(\rho,u) (\mu^\eps(\rho) u_x)_x}_{\NIII} 
\end{equation*}
Multiply this equation by $\sign(x-y)$ and integrate it on $(t,x,y) \in [0,T] \times \RR \times K$.

By~(\ref{lemmaboundbis}), the contribution of $\NI$ can be bounded by $\displaystyle \dots \lesssim C( \Escr_0,\mathscr{M}_0)$. Still by~\eqref{lemmaboundbis}, the contribution of $\NII$ is
$$ \dots \gtrsim - \int_0^T \int_K \left( \rho |u|^{2+a} + \rho^{\gamma+a\theta} \right) \,dy\,ds.$$

We are left with $\NIII$. Its contribution can be bounded with the help of the bounds~\eqref{lemmaboundbis} and Cauchy-Schwarz inequality:
\begin{equation*}
\begin{split}
& \left| \int_0^T \int_\RR \int_K \sign(x-y) \, \NIII \, dy\, dx\, ds \right|\\
& \qquad \qquad  \lesssim \int_0^T \int_\RR \left( \mathbbm{1}_K (\rho^{a\theta} + |u|^a)\mu(\rho)|u_x| 
+ \mu(\rho) \rho^{a\theta - 1} |\rho_x| |u_x|
+ \mu(\rho) \rho^{(a-1)\theta} \langle \frac{u}{\rho^\theta} \rangle^{a-1} u_x^2 \right)\, dtdx\\
& \qquad \qquad  \lesssim \int_0^T \int_\RR \left( \underbrace{\mu(\rho)|u_x|^2}_{\NIII_1}
+ \underbrace{\mu(\rho) \rho^{(a-1)\theta} \langle \frac{u}{\rho^\theta} \rangle^{a-1} u_x^2}_{\NIII_2}
+ \underbrace{\mu(\rho) \rho^{2a\theta - 2}|\rho_x|^2}_{\NIII_3}
+ \underbrace{\mathbbm{1}_K \mu(\rho) |u|^{2a}}_{\NIII_4}
+ \underbrace{\mathbbm{1}_K \mu(\rho) \rho^{2a\theta}}_{\NIII_5} \right)\, dtdx
\end{split}
\end{equation*}
We examine each of the terms appearing above: $\NIII_1$ is obviously bounded by the energy; $\NIII_2$ was controlled earlier;
$\NIII_3$ was already treated in Step 2; and the term $\NIII_5$ can be controlled as in Proposition~\ref{grebe}. All in all,
$$
\int_0^T \int \left( \NIII_1 + \NIII_2 + \NIII_3 + \NIII_5\right)\, dx\,dt \lesssim C(\Escr_0,\Escrt_0,T,K).
$$
Finally, H\"older's inequality gives
$$
\int_0^T  \int  \NIII_4 \, dx\,dt \lesssim \left( \int_0^T \int \mu(\rho)^{1/a} |u|^{2}\, dx\,dt \right)^a \lesssim \left( \int_0^T \int \rho |u|^2 \, dx\,dt \right)^a \lesssim T \Escr_0,
$$
where the prior-to-last inequality follows from the assumption $\mu(\rho) \lesssim \rho^a$. 
\end{proof}


\subsection{Existence theory for cavitating solutions ($\kappa>0$ and $\DD=\RR$)}

We now consider the full system with both viscosity and capillarity terms. 

\begin{theorem}[Cavitating finite energy solutions to the Navier-Stokes-Korteweg system] 
\label{guillemot3}
Assume the viscosity and capilarity coefficients $\mu,\kappa > 0$
satisfy the strong coercivity condition (SC)
and, that additionally, the pair $(1,\kappa)$ also satisfies (SC) and the following growth conditions hold: 
\begin{equation}
\label{mesange}
\mu(\rho) \lesssim \rho^{2/3}, \qquad  \quad \kappa(\rho) \lesssim \frac{\mu(\rho)^2}{\rho^3}, \qquad \mbox{for small } \rho >0.
\end{equation}
Consider the Navier-Stokes-Korteweg system \eqref{eq:101} posed on the real line with $\rho_\star \geq 0$
with prescribed initial data $(\rho_0,u_0)$ with finite physical energy 
$\Escr_0$ and effective energy $\Escrt_0$. Then, there exists a global solution of (NSK) (in the sense of distributions) such that $(\EDscr)$ holds
with, furthermore, $\rho \in L^\infty_{t,x}$ and $\mu(\rho) \in \cap \mathcal{C}^{1/4}_{t,x}$.
\end{theorem}

\begin{definition} Solutions of (NSK) are understood in the sense of distributions.
As already explained in the case of Navier-Stokes, the uniform estimate $(\EDscr)$ ensures that $\rho$, $\rho u$, $\rho u^2$, $p(\rho)$ and $\mu(\rho) u_x$ are locally integrable.
It thus simply remains to give a meaning to the capillarity term $K$. It suffices to write it
$$
K[\rho] = \big( \rho \kappa(\rho) \rho_x \big)_x  - \frac{1}{2} \, \big( \rho \kappa'(\rho) + 3 \kappa(\rho) \big) \, \rho_x^2,
$$
and to notice that, since $\rho$ is bounded, the finiteness of $\mathscr{E}$ implies that $\rho \kappa(\rho) \rho_x$ and $\big( \rho \kappa'(\rho) + 3 \kappa(\rho) \big) \, \rho_x^2$ are locally integrable.
\end{definition}

\begin{proof}
{\bf 1. Approximation by the non-cavitating case.}
We replace the initial data by $\left( \rho_0 + 1/n , u_0 \right)$ and 
we replace the viscosity coefficient by $\mu(\rho)+ 1/n$. This gives a sequence of solutions $(\rho^n,u^n)$ which, we would like to show, is compact and converges to a solution of the
desired equation. It can be done in a very similar way to the proof of Theorem~\ref{guillemot3}, we only describe here two new 
technical ingredients: 
the derivation of improved integrability for the velocity $u$ (which was proved in Lemma~\ref{loriot} in the purely viscous case, provided additional integrability already held initially), 
and the passage to the limit in the capillarity term.

\vskip.15cm 

\noindent
{\bf 2. Higher-integrability property for the velocity.} We only treat here the case $\rho_\star > 0$, the case $\rho_\star = 0$ can be dealt with in a similar fashion.
We wish to prove that, for any compact $K$, $\int_0^T \int \rho^n |u^n|^3 \, dtdx$ is uniformly bounded in $n$. This is quite close to the statement of
Proposition~\ref{grebe} (with $\eps=1$), and indeed the proofs are almost identical; note that the assumption~(\ref{mesange}) is needed here. 
The only difference is a new term in the last line of~(\ref{circaete}), below, specifically
$$
\int_0^T \int \frac{1}{n} |u|^2 \mathbbm{1}_K \, dtdx.
$$ 
We indicate here how to treat this error term (following an argument in~\cite{CP}) and we refer to Proposition~\ref{grebe} for the rest of the proof and the more 
general context. Observe first that the bound on $\Escr$ implies the existence of $R(\Escr_0),r(\Escr_0)$ such that: 
for any $t,x_0$, there exists a set $I$ of length one, contained in $[x_0-R,x_0+R]$ such that $\rho(t,x) \geq r$ on $I$. 
Still by the bound on $\Escr$, this implies that there exists $R(\Escr_{0}),L(\Escr_{0})$ such that:
for any $t,x_0$, there exists $x \in [x_0-R,x_0+R]$ such that $|u(t,x)| \leq L$. Now, pick any $x_0$ and find an associated $x$ with this property.
Then
$$
|u(x_0)| \leq |u(x)| + \int_{x_0}^{x} |u_x| \leq L(\Escr_0) + R(\Escr_0)^{1/2} \left( \int |u_x|^2 \right)^{1/2} 
\lesssim C(\Escr_0) + \sqrt{n} C(\Escr_0).
$$
This $L^\infty$ bound on the velocity enables us to treat the term we are discussing:
$$
\int_0^T \int \frac{1}{n} |u|^2 \mathbbm{1}_K \, dtdx 
= \int_0^T \mathbbm{1}_K \frac{1}{n} |u|^2\, dtdx \lesssim \frac{1}{n} \left( C(\Escr_0) + n C(\Escr_0)\right) \lesssim C(\Escr_0).
$$

\vskip.15cm 

\noindent
{\bf 3. Passing to the limit in the capillarity term.} Recall that, as was established in Theorem~\ref{guillemot2}, 
$\rho^n$ converges to $\rho$ in $\mathcal{C}^{1/4}_{t,x}$ away from $\{ \rho = 0 \}$.
Taking advantage of the bound on $\rho_{xx}^2$ which becomes available since $\kappa>0$, it is easy to show, via the Aubin-Lions lemma, that
$\rho^n$ converges to $\rho$ in $L^2_t H^1_x$ away from $\{\rho = 0\}$.

The capillarity term reads $K[\rho]_x$, with $K[\rho] = \rho \kappa(\rho) \rho_{xx} + \frac{1}{2} \left( \rho \kappa'(\rho) - \kappa(\rho) \right) \rho_x^2$.
We want to show that $K[\rho^n]$ converges as $n \to \infty$ to $K[\rho]$ in the sense of distributions. We focus on the hardest term and show that
$\rho^n \kappa(\rho^n) \rho_{xx}^n$ converges to $\rho \kappa(\rho) \rho_{xx}$ . Write first
$$
\rho^n \kappa(\rho^n) \rho_{xx}^n = ( \rho^n \kappa(\rho^n) \rho_x^n )_x - (\rho \kappa)'(\rho^n) (\rho_x^n)^2.
$$
Once again, we focus on the most difficult term, namely $(\rho \kappa)'(\rho^n) (\rho_x^n)^2$. Taking $\chi$ to be the indicator function of the unit ball,
and $K$ a compact set,
\begin{equation*}
\begin{split}
& \int_0^T \int_K \left| \kappa(\rho^n) (\rho^n_x)^2 - \kappa(\rho) \rho_x^2 \right| \, dtdx \\
& \lesssim \int_0^T \int_K \left(1- \chi\left(\frac{\rho}{\eps} \right)\right)  \left| \kappa(\rho^n) (\rho^n_x)^2 - \kappa(\rho) \rho_x^2 \right| \, dtdx 
+ \int_0^T \int_K \chi\left(\frac{\rho}{\eps} \right) \kappa(\rho) \rho_x^2 \, dtdx \\
& \quad + \int_0^T \int_K \chi \left(\frac{\rho}{\eps} \right) \kappa(\rho^n) (\rho^n_x)^2 \, dtdx.
\end{split}
\end{equation*}
The first term in the above right-hand side goes to zero as $n$ goes to infinity, since $\rho^n$ converges to $\rho$ in $L^2 H^1$ away from $\{\rho = 0 \}$. 
As for the second term, it can be bounded using H\"older's inequality and~(\ref{mesange})
\begin{equation*}
\begin{split}
\int_K \int_0^T \int \chi\left(\frac{\rho}{\eps} \right)\kappa(\rho) \rho_x^2 \, dtdx 
& \lesssim \left( \int_0^T \int_K \chi\left(\frac{\rho}{\eps} \right)^2 \kappa(\rho)^2 \rho_x^4 \, dtdx \right)^{1/2} 
\\
& \lesssim \sup_{\rho \in [0,\eps]} \left( \rho^3 \frac{\kappa(\rho)}{\mu(\rho)} \right)^{1/2} 
\left( \int_0^T \int_K \frac{\kappa(\rho) \mu(\rho)}{\rho^3} \rho_x^4 \, dtdx \right)^{1/2}
\\
& 
\lesssim  \sup_{\rho \in [0,\eps]} \left( \rho^3 \frac{\kappa(\rho)}{\mu(\rho)} \right)^{1/2} \Escrt_0, 
\end{split}
\end{equation*}
in which the upper bound now tends to zero with $\eps \to 0$, thanks to 
the assumptions $\kappa(\rho) \lesssim \frac{\mu(\rho)^2}{\rho^3}$ and $\mu(\rho) \rightarrow 0$ as $\rho \rightarrow 0$.
Finally, the third term can be bounded similarly since for $n$ big enough, $\rho^n < \frac{\eps}{2}$ if $\rho <\eps$. Therefore,
as $n \to +\infty$ 
$$
\int_0^T \int_K \left| \kappa(\rho^n) (\rho^n_x)^2 - \kappa(\rho) \rho_x^2 \right| \, dtdx \to 0,
$$
which is the desired result.
\end{proof}


\section{Finite energy solutions to the Euler system} 
\label{sec:4}

\subsection{Higher integrability property for the pressure}
\label{hipt}

We consider in this section solutions of the system (NSK)$_\epsilon$ satisfying the uniform bound~(\ref{cigogne2}), where we recall that 
$$
\mu^\eps (\rho) = \epsilon \mu(\rho), \qquad \quad \kappa^\eps(\rho) = \delta(\epsilon) \kappa(\rho).  
$$
For simplicity in the notation, we drop the superscript $\epsilon$ and simply write $(\rho,u)$ instead of $(\rho^\eps,u^\eps)$.
In this section, we assume the following: 
\begin{itemize}
\item The equation is set on $\RR$.
\item The conditions (TC) (defined in~(\ref{defTC})) and (SC) (defined in~(\ref{eq:203bis})) hold.
\item The fluid is perfect and polytropic, that is, $p(\rho) = \frac{(\gamma-1)^2}{4\gamma} \rho^\gamma$ with $\gamma>1$.
\end{itemize}
As it is usual, we set  $\theta : = \frac{\gamma-1}{2}$. 

\begin{lemma}
\label{papillon}
Under the above assumptions, there exists a constant\footnote{The constant also depends on $\mu$, $\kappa$, but we consider these as fixed.} 
$C(\Escr_0,\Escrt_0) >0$ such that
\be
\eps \mu(\rho) \langle \rho \rangle^\theta \leq C(\Escr_0,\Escrt_0) 
\label{fourmi0}
\ee
with
\be
\label{fourmi1}
\eps \frac{\mu(\rho) \rho}{\langle p(\rho) \rangle} = C(\Escr_0,\Escrt_0) \, o(1),  
\ee
where $o(1)$ denotes here a function tending to zero when $\eps \to 0$.
\end{lemma}

\begin{proof} Thanks to the conservation property for the normalized energy, there exists $L(\Escr_0)$ such that for any $x_0$, $\inf_{[x_0-1,x_0+1]} \rho(x) \leq L$, therefore
\be
\label{fourmi2}
\inf_{[x_0-R,x_0+R]} \mu(\rho) \langle \rho \rangle^\theta \lesssim C(\Escr_0).
\ee
On the other hand, $\epsilon \partial_x \left(\epsilon \mu(\rho) \langle \rho \rangle^\theta \right)$ is uniformly bounded in $L^1_x$. On the set where $\rho\leq1$, this is clear. 
On the set where $\rho\geq 1$, by (TC) and Cauchy-Schwarz inequality,
\be
\label{fourmi3}
\aligned
\left| \int \partial_x \left( \mu(\rho) \langle \rho \rangle^\theta \right) \, dx \right| 
& \lesssim \int \mu(\rho) \rho^{\theta-1} \rho_x \, dx
\\
& \leq \frac{1}{\eps} \left( \int \frac{\eps^2 \mu(\rho)^2 |\rho_x|^2}{\rho^3} \right)^{1/2} \left( \int \rho^{\gamma} \right)^{1/2}
\leq \frac{C(\Escr_0,\Escrt_0)}{\eps},
\endaligned
\ee
Gathering the two previous inequalities gives~(\ref{fourmi0}), from which~(\ref{fourmi1})
follows.
\end{proof}

\begin{proposition}[Higher integrability of the pressure] 
Fix a compact set $K$. Under the assumptions recalled at the beginning of the present section,
for any $T>0$, there exists a constant $C(K,T,\Escr_0,\Escrt_0)$ such that
$$
\int_0^T \int_K\left( \rho p(\rho) + \rho (\rho {\kappa^\eps}'(\rho) + 5 \kappa^\eps(\rho)) \rho_x^2 \right) \, dxdt \leq C(K,T,\Escr_0,\Escrt_0).
$$
\end{proposition}

\begin{proof} \vskip.15cm  \noindent \textbf{1. The key commutator identity.}
Start from the conservation of momentum equation, and integrate it over an interval $[y,x]$ to obtain
\begin{align}
\nonumber
\NI := &
\left( \rho u^2 + p(\rho) - \mu^\eps(\rho) u_x - (\rho \kappa^\eps(\rho)\rho_{xx} + \frac{1}{2}(\rho {\kappa^\eps}'(\rho) - \kappa^\eps(\rho))(\rho_x)^2) \right)(y)
\\
\nonumber
& = \frac{d}{dt} \int_y^x (\rho u) + \left(
\rho u^2 + p(\rho) - \mu^\eps(\rho) u_x - (\rho \kappa^\eps(\rho)\rho_{xx} + 
\frac{1}{2}(\rho {\kappa^\eps}'(\rho) - \kappa^\eps(\rho))(\rho_x)^2) \right)(x)
\\
&=:\NII.
\nonumber 
\end{align}
Next, pick a smooth, nonegative, compactly supported function $\phi$ equal to 1 on $K$,
multiply the above by $\rho(x) \phi(x) \phi(y)$, and integrate over $y,x \in \RR$, and $t \in [0,T]$. We 
estimate separately the contributions $\NI, \NII$.   

\vskip.15cm  

\noindent \textbf{2. Contribution $\NI$.}
Using successively  Cauchy-Schwarz inequality, the condition (TC), and Lemma~\ref{papillon}, it can be bounded as follows:
\begin{equation*}
\begin{split}
& \left| \int_0^T \iint \phi(y) \phi(x)\rho(x) \, \NI \, dx \, dy \, dt \right| 
\\
& \lesssim \int_0^T \int \rho(x) \phi(x) \phi(y) \left( \rho u^2 + p(\rho) + \mu^\eps (\rho) u_x^2 + \kappa^\eps(\rho) \rho_x^2 + \mu^\eps(\rho) + \rho \kappa^\eps(\rho) |\rho_{xx}| \right)(y)\, dx\,dy\,dt  \\
& \lesssim C(T,\mathscr{E}_0) + \int_0^T \rho(x) \phi(x) \phi(y) \left( \langle \rho \rangle^{-\theta} + \frac{\mu^\eps(\rho) \kappa^\eps(\rho)}{\rho} \rho_{xx}^2 \right)\, dx\,dy\,dt \leq C(K,T,\Escr_0,\Escrt_0).
\end{split}
\end{equation*}

\vskip.15cm  

\noindent \textbf{3. Contribution $\NII$.} 
This term reads
\begin{equation*}
\begin{split} 
\NII= & \underbrace{ \frac{d}{dt} \left( \int_y^x (\rho u)\right) + (\rho u^2)(x)}_{A_1} + \underbrace{ p(\rho(x))}_{A_2} - 
\underbrace{\left( \mu^\eps(\rho) u_x \right) (x)}_{A_3} 
- \underbrace{[\rho \kappa^\eps(\rho)\rho_{xx} - \frac{1}{2}(\rho {\kappa^\eps}'(\rho) - \kappa^\eps(\rho))(\rho_x)^2](x)}_{A_4}.
\end{split}
\end{equation*}
Recall that we perform the following manipulation: multiply by $\rho(x) \phi(x) \phi(y)$, and integrate over $y,x \in \RR$, and $t \in [0,T]$.
Integrating by parts in $t$, using the equation of conservation of mass, and then integrating by parts in $x$, gives
\begin{equation*}
\begin{split}
& \int_0^T \iint \rho(x) \phi(x) \phi(y) \, A_1 \, dt\,dx \,dy
\\
& = \left. \iint \rho(x) \phi(x) \phi(y) \int_y^x (\rho u) \,dy\, dx\right|_0^T - \int_0^T \iint \left( \phi(x) \phi(y) \partial_t \rho(x) \int_x^y (\rho u)
-  \phi(x) \phi(y) (\rho^2 u^2)(x) \right) \, dt\,dx \,dy 
\\
& = \left. \iint \rho(x) \phi(x) \phi(y) \int_y^x (\rho u) \,dy\, dx\right|_0^T 
\\
& \quad - \int_0^T \iint \left( - \partial_x (\rho(x)u(x)) \phi(x)\phi(y) \int_y^x (\rho u)
- \phi(y)\phi(x) \rho(x)^2 u(x)^2 \right)\, dt\,dx \,dy 
\\
& = \left. \iint \rho(x) \phi(x) \phi(y) \int_y^x (\rho u) \,dy\, dx\right|_0^T + \int_0^T \iint (\rho u)(x) \phi'(x) \phi(y) \int_y^x (\rho u)\, dt\,dx \,dy.
\end{split}
\end{equation*}
By H\"older's inequality, we find thus
\be
\nonumber
\left| \int_0^T \iint \rho(x) \phi(x) \phi(y) A_1 \,dy\, dtdx \right| \lesssim C(K,T,\Escr_0),
\ee
while the term $A_2$ gives immediately
\be
\nonumber
\int_0^T \iint \rho(x) \phi(x) \phi(y) \, A_2 \,dy\, dtdx = c_0 \int_0^T \int \phi \rho p(\rho) \, dtdx,
\ee
where $c_0 = \int \phi$.
Applying successively Cauchy-Schwarz inequality, the condition (TC), and~(\ref{fourmi1}), one obtains
\begin{equation*}
\begin{split}
\left| \int_0^T \iint \rho(x) \phi(y) \phi(x) \, A_3 \,dy\, dtdx \right| 
& \lesssim \frac{1}{\zeta} \int_0^T \int \phi \mu^\eps(\rho) u_x^2 \, dtdx + \zeta \int_0^T \int \phi \mu^\eps(\rho) \rho^2 \, dtdx \\
& \lesssim C(\Escr_0) + \zeta C(\Escr_0,\Escrt_0) \int_0^T \int \phi \rho p(\rho) \, dtdx ,
\end{split}
\end{equation*}
where $\zeta$ is a small constant which will be determined shortly.
Finally, a few integrations by parts yield
\begin{equation*}
\begin{split}
\iint \rho(x) \phi(x) \phi(y) \, A_4 \, dx\,dy\,ds & = \frac{1}{2} c_0 \int_0^T \int \rho(x) \rho_x^2(x) \left( \rho(x) {\kappa^\eps}'(\rho(x)) + 5 \kappa^\eps(\rho(x)) \phi(x) \right) \, dx \\
& \qquad + c_0 \int_0^T \int \phi'(x) \rho(x)^2 \kappa^\eps (x) \rho_x(x)\, dt\, dx.
\end{split}
\end{equation*}
We bound using Cauchy-Schwarz inequality, (TC) and Lemma~\ref{papillon}
\begin{equation*}
\begin{split}
\left| \int_0^T \int \phi'(x) \rho(x)^2 \kappa^\eps (\rho(x)) \rho_x(x)\, dtdx \right| & \lesssim \int_0^T \int |\phi'(x)| \kappa^\eps (x) \rho_x(x)^2 \, dt\, dx
+ \int_0^T \int |\phi'(x)| \kappa^\eps (\rho(x)) \rho(x)^4 \, dtdx \\
& \lesssim C(\Escr_0) + \int_0^T \int \phi(x) \mu^{\eps 2}(\rho(x)) \rho(x) \, dtdx \lesssim C(\Escr_0).
\end{split}
\end{equation*}
Gathering the previous results gives
\begin{equation*}
\begin{split}
& \int_0^T \iint \phi(y) \phi(x) \rho(x) \, \NII 
\, dydtdx 
- c_0 \int_0^T \int \left( \rho p(\rho) + \frac{1}{2} \rho (\rho {\kappa^\eps}'(\rho) + 5 \kappa^\eps(\rho)) \rho_x^2 \right)(x)\, dtdx \\
& \lesssim \zeta C(\Escr_0,\Escrt_0) \int_0^T \int \rho p(\rho)\, dtdx + C(\Escr_0,\Escrt_0).
\end{split}
\end{equation*}
Taking $\zeta$ to be small enough and combining~this inequality with the contributions $\NI$ and $\NII$ 
gives the desired result.
\end{proof}


\subsection{Higher integrability property for the velocity} 

\label{sec:43}

General entropy pairs are obtained by integrating the fundamental entropy kernel 
\bel{eq:chidef}
\chi(\rho,u;v) = \left(\rho^{2\theta} - (v-u)^2 \right|_+^\lambda
\ee
with $\theta = \frac{\gamma-1}{2}$ and $\lambda = \frac{3-\gamma}{2(\gamma-1)}$ 
and the fundamental entropy flux kernel 
\bel{eq:etadef} 
\sigma(\rho,u;v)= \left( \theta v + (1-\theta) u \right) \left(\rho^{2\theta} - (v-u)^2 \right|_+^\lambda
\ee
against an arbitrary function $\psi=\psi(v)$, namely: 
\bel{eq:4200}
\eta^\psi(\rho,u) := \int_\RR \chi(\rho,u;v) \, \psi(v) \, dv, 
\quad 
\qquad q^\psi(\rho,u) := \int_\RR \sigma(\rho,u;v) \psi(v) \,dv.
\ee
With the choice $\psi = v |v|$, we will simply denote
$$
\etat := \eta^\psi,
\qquad \widetilde{q} := q^\psi.
$$

It is convenient in the following to consider that all functions depend either on $u,\rho$, or on $m,\rho$, depending on the context.
When derivatives are taken, we adopt the following convention: derivatives in $m,u$ are always taken by keeping $\rho$ constant, while derivatives in $\rho$ are always taken by 
keeping $u$ constant.

The following pointwise bounds hold for $\etat$ and $\widetilde{q}$: 
\bel{lemmabound}
\aligned
& \widetilde{q} \gtrsim \rho |u|^3 + \rho^{\gamma + \theta}, 
\qquad \qquad 
&& \left| \etat \right| \lesssim \rho u^2 + \rho^\gamma,
 \\
& \left| \etat_m \right| \lesssim |u| + \rho^\theta,
&& \left| \etat_{m \rho} \right| \lesssim \rho^{\theta-1}, 
&&& \left| \etat_{mu} \right| \lesssim 1, 
\qquad
&&
\endaligned
\ee 
These bounds appeared first in Lions, Perthame and Tadmor~\cite{LPT}, and were already used in \cite{LW} and next in~\cite{CP}. 
In \cite{GL2}, these estimates will be checked to hold for a broad class of pressure functions, so that the estimate in the following proposition will also be established for general 
pressure laws. 

\begin{proposition}
\label{grebe}
[Higher integrability of the velocity]
Under the assumptions stated at the beginning of the present section, and assuming furthermore that
$$
\mu(\rho) \lesssim \rho^{2/3} \qquad \mbox{for $\rho$ small}
$$
for any $T>0$ and every compact set $K$, there exists a constant $C(T, K,\Escr_0,\Escrt_0)$ such that
$$
\int_0^T \int_K \big( \rho |u|^3 + \rho^{\gamma+\theta} \big) \, dtdx \leq C(T, K,\Escr_0,\Escrt_0).
$$
\end{proposition}

\begin{proof}
The conservation law for the entropy $\etat$ reads
\begin{equation*}
\underbrace{\partial_t \etat(\rho,u)}_{\NI} + \underbrace{\partial_x \widetilde{q}(\rho,u)}_{\NII} 
= 
\underbrace{\etat_m(\rho,u) (\mu^\eps(\rho) u_x)_x}_{\NIII} 
+ \underbrace{\etat_m(\rho,u) \left(
\rho \kappa^\eps(\rho)\rho_{xx} + \frac{1}{2}(\rho {\kappa^\eps}'(\rho) - \kappa^\eps(\rho))(\rho_x)^2\right|_x}_{\NIV}. 
\end{equation*}
Multiply this equation by $\sign(x-y)$ and integrate it on $(t,x,y) \in [0,T] \times \RR \times K$.
We examine separately the contributions of $\NI$, $\NII$, $\NIII$, and $\NIV$.

\vskip.15cm 
\noindent
\textbf{Contribution $\NI$.}
By~(\ref{lemmabound}), this term can be bounded by
\begin{equation*}
\label{eq:601}
\begin{split}
\left| \int_0^T \int_{\RR} \int_K \sign(x-y) \NI \, dtdxdy  \right|
\lesssim \int_K \int_\RR |\eta(T,x)| \, dx\,dy + \int_K \int_\RR |\eta(0,x)| \, dx \,dy \\
\lesssim \int_\RR \left(\rho u^2 + \rho^\gamma \right)(T,x) \, dx + \int_\RR \left(\rho u^2 + \rho^\gamma \right)(0,x) \, dx
\leq 2 \Escr_0.
\end{split}
\end{equation*}

\vskip.15cm 
\noindent
\textbf{Contribution $\NII$.}
By~\eqref{lemmabound}, this term contributes 
$$
\int_0^T \int_\RR \int_K \sign(x-y) \, \NII \, dx\, dy\, ds = - \int _0^T \int_K \widetilde{q}(\rho(s,y),u(s,y))\,dy\, ds 
\lesssim - \int_0^T \int_K \left( \rho |u|^3 + \rho^{\gamma+\theta} \right) \,dy\,ds.
$$

\vskip.15cm 
\noindent
\textbf{Contribution $\NIII$.} By integration by parts, this term can be written as 
\be\nonumber
\begin{split}
\int_0^T \int_\RR \int_K \sign(x-y) \, \NIII \, dx\,dy\, ds 
& = - \int_0^T \int_\RR \int_K \sign(x-y) \left( \etat_{m \rho} \rho_x + \etat_{m u} u_x \right) (\mu^\eps(\rho) u_x) \, dx\,dy\,ds
\\
&\quad + \int_0^T \int_K \etat_m(\rho,u) \mu^\eps(\rho) u_x \,dy\, ds , 
\end{split}
\ee
which we bound using~\eqref{lemmabound}, and Cauchy-Schwarz inequality:
\be
\label{circaete}
\begin{split}
&\left| \int_0^T \int_\RR \int_K \sign(x-y) \, \NIII \, dx\,dy\,ds \right|
\\
& \lesssim \int_0^T \int \left( \mu^\eps(\rho) \rho^{\theta-1} |\rho_x| |u_x| + \mu^\eps(\rho) u_x^2 + \mathbbm{1}_K (|u| + \rho^\theta) \mu^\eps(\rho) |u_x|  \right) \, dsdx \\
& = \int_0^T \int \Bigg( 
\underbrace{\frac{1}{\zeta} \mu^\eps(\rho) |u_x|^2}_{\NIII_1} 
+ \underbrace{\mu^\eps(\rho) \rho^{2\theta-2} |\rho_x|^2}_{\NIII_2} 
+ \underbrace{\zeta \mathbbm{1}_K \,\mu^\eps(\rho) |u|^2}_{\NIII_3} 
+ \underbrace{\mathbbm{1}_K \mu^\eps(\rho) \rho^{2\theta} }_{\NIII_4} \Bigg)\, dsdx.
\end{split}
\ee
Here, $\zeta$ is a small constant whose value will be fixed in the following. We now bound one by one the terms $\NIII_1$ to $\NIII_4$. The first one is easy:
$$
\int_0^T \int \NIII_1 \, dsdx =  \frac{1}{\zeta} \int_0^T  \int \mu^\eps(\rho) |u_x|^2\, dsdx  \lesssim\Escr_0.
$$
Dealing with the second term is not more difficult, since  
$$
\int_0^T \int \NIII_2 \, dtdx = \int_0^T \int \mu^\eps(\rho) \rho^{2\theta-2} |\rho_x|^2 dsdx = \int_0^T \int \frac{\mu^\eps(\rho) p'(\rho)}{\rho^2} |\rho_x|^2 dsdx\lesssim \Escrt_0.
$$
Next, for the third term, we resort to~(\ref{fourmi0}) and the assumption $\mu(\rho) \lesssim \rho^{2/3}$ for $\rho$ small:
\be
\begin{split}
\int_0^T \int \NIII_3 \, dtdx & = \zeta \int_0^T \int \mathbbm{1}_K \mu^\eps(\rho) |u|^2 \, dtdx \\
& = \zeta \int_0^T \int_{\{\rho \leq 1\}} \mathbbm{1}_K \mu^\eps(\rho) |u|^2 \, dtdx + \zeta \int_0^T \int_{\{\rho \geq 1\}} \mathbbm{1}_K \mu^\eps(\rho) |u|^2 \, dtdx \\
& \leq \zeta \int_0^T \int_{\{\rho \leq 1\}} \mathbbm{1}_K \mu^\eps(\rho) \rho^{-2/3} \rho^{2/3} |u|^2 \, dtdx 
+ \zeta C(\Escr_0,\Escrt_0) \int_0^T \int_{\{\rho \geq 1\}} \mathbbm{1}_K \langle \rho \rangle^{-\theta} |u|^2 \, dtdx \\
& \leq \zeta \eps \int_0^T \int_K \rho |u|^3\,dtdx + \zeta C(K,T,\Escr_0,\Escrt_0).
\end{split}
\ee
Finally, using once again~(\ref{fourmi0}) gives
\be
\begin{split}
\int_0^T \int \NIII_4 \, dtdx  =  \int_0^T \int_K \mu^\eps(\rho) \rho^{2\theta} \, dtdx 
\leq C(\Escr_0,\Escrt_0) \int_0^T \int_K \langle \rho \rangle^{\theta} \, dtdx \leq C(T,K,\Escr_0,\Escrt_0).
\end{split}
\ee

\vskip.15cm 
\noindent
\textbf{Contribution $\NIV$.} By integrating by parts, this term can be written in the form 
\be
\nonumber
\begin{split}
&\int_0^T \int_\RR \int_K \sign(x-y) \, \NIV \, ds dx dy
\\
&= - \int_0^T \int_\RR \int_K \sign(x-y)
\left( \etat_{m \rho} \rho_x + \etat_{m u} u_x \right) 
\left(\rho \kappa^\eps(\rho)\rho_{xx} + \frac{1}{2}(\rho {\kappa^\eps}'(\rho) - \kappa^\eps(\rho))(\rho_x)^2\right) \, dsdxdy
\\
&\quad + \int_0^T \int_K \etat_m(\rho(s,y),u(s,y))  
\left(\rho \kappa^\eps(\rho)\rho_{xx} + \frac{1}{2}(\rho {\kappa^\eps}'(\rho) - \kappa^\eps(\rho))(\rho_x)^2\right)  \, dsdy, 
\end{split}
\ee
which we bound using successively \eqref{lemmabound} and Cauchy--Schwarz inequality:
\begin{equation*}
\begin{split}
&\left| \int_0^T \iint \sign(x-y) \, \NIV \, dx\,dy\,ds \right|
\\
&\lesssim \int_0^T \int \left(\mathbbm{1}_K(|u| + \rho^\theta) + \rho^{\theta-1} |\rho_x| + |u_x| \right) 
\left( \rho \kappa^\eps |\rho_{xx}| + \kappa^\eps \rho_x^2 \right) \, dsdx 
\\
&\lesssim \int_0^T \int \left( \zeta \mathbbm{1}_K \mu^\eps(\rho) u^2 + \mu^\eps(\rho) u_x^2 + \mathbbm{1}_K \mu^\eps(\rho) \rho^{2\theta} + \mu^\eps(\rho) \rho^{2\theta-2} \rho_x^2 
+ \frac{1}{\zeta} \frac{ \kappa^{\eps 2}(\rho)\rho^2}{\mu^\eps(\rho)} (\rho_{xx})^2 + \frac{1}{\zeta} \frac{ \kappa^{ \eps 2}(\rho)}{\mu^\eps(\rho)} (\rho_x)^4 \right) \, dsdx.
\end{split} 
\end{equation*}
The first four terms on the right-hand side have already been bounded when treating $\NIII$, and they contribute
$$
\int_0^T \int \left(\mathbbm{1}_K \mu^\eps(\rho) u^2 + \mu^\eps(\rho) u_x^2 + \mathbbm{1}_K \mu^\eps(\rho) \rho^{2\theta} + \mu^\eps(\rho) \rho^{2\theta-2} \rho_x^2  \right) \, dsdx \lesssim
C(T,K,\Escr_0,\Escrt_0) + \zeta \eps \int_0^T \int_K \rho |u|^3\, dsdx.
$$
As for the last two terms, they can be dealt with using (TC):
\be \begin{split}
\frac{1}{\zeta} \int_0^T \int \left( \frac{\kappa^{\eps 2}(\rho)\rho^2}{\mu^\eps(\rho)} (\rho_{xx})^2 + \frac{\kappa^{\eps 2}(\rho)}{\mu^\eps(\rho)} (\rho_x)^4 \right) \, dtdx
& \lesssim \frac{1}{\zeta} \int_0^T \int \left( \frac{\mu^\eps(\rho) \kappa^\eps(\rho)}{\rho} (\rho_{xx})^2 + \frac{\mu^\eps(\rho) \kappa^\eps(\rho)}{\rho^3} (\rho_x)^4 \right) \, dtdx 
\\
& \lesssim \frac{1}{\zeta} (\Escr_0 + \Escrt_0). 
\end{split} \ee
The desired conclusion follows by gathering the contributions $\NI$ to $\NIV$ and taking $\zeta$ small enough. 
\end{proof}


\subsection{Existence with finite energy solutions to the Euler equations} 
\label{sec:5}

Consider solutions of (NSK)$_\epsilon$ for which~(\ref{cigogne2}) holds. We gather below all the uniform estimates (uniform in $\epsilon$) which have been proved so far (with the subscript $\eps$ removed for ease of reading): 
\be
\label{eq:650} 
\aligned
& \left\| \frac{1}{2} \rho u^2 + \rho e(\rho) + \frac{1}{2} \kappa^\eps(\rho) \rho_x^2 \right\|_{L^\infty_t L^1_x} \leq C(\Escr_0,\Escrt_0),
\\
& \left\| \mu^\eps(\rho) u_x^2 \right\|_{L^1_{t,x}} \leq C(\Escr_0,\Escrt_0),
&&
 \left\| \frac{ \mu^\eps(\rho)^2}{\rho^3} (\rho_x)^2 \right\|_{L^\infty_t L^1_x}\leq C(\Escr_0,\Escrt_0),
\\
& \left\| \frac{\mu^\eps(\rho) p'(\rho)}{\rho^2} (\rho_x)^2 \right\|_{L^1_{t,x}} \leq C(\Escr_0,\Escrt_0),
&&
\left\| \frac{ \mu^\eps(\rho) \kappa^\eps(\rho)}{\rho}(\rho_{xx})^2
+ \frac{ \mu^\eps(\rho) \kappa^\eps(\rho)}{\rho^3}(\rho_x)^4 \right\|_{L^1_{t,x}} \leq C(\Escr_0,\Escrt_0),
\\
& \left\| \mu^\eps(\rho) \langle \rho \rangle^{\theta} \right\|_{L^\infty_{t,x}} \leq  C(\Escr_0,\Escrt_0),
&&
 \left\| \frac{ \mu^\eps(\rho) \rho}{p(\rho)} \right\|_{L^\infty_{t,x}} \leq C(\Escr_0,\Escrt_0)o(1),
\\
& \left\| \rho p(\rho) + \rho \left( \rho {\kappa^\eps}'(\rho) + 5 \kappa^\eps(\rho) \right) (\rho_x)^2 \right\|_{L^1_{t,x \operatorname{loc}}} \leq C(\Escr_0,\Escrt_0),
\\
&
 \left\| \rho |u|^3 + \rho^{\gamma + \theta} \right\|_{L^1_{t,x \operatorname{loc}}}\leq C(\Escr_0,\Escrt_0).
\endaligned
\ee
We now turn to the proof of Theorem~\ref{theo:12}, which relies 
on the energy and higher-order integrability estimates stated in \eqref{eq:650}
in combination with the compactness framework established in \cite{LW} for polytropic fluids. 

The entropy--entropy flux pairs $(\eta^\psi,q^\psi)$ were defined in Section~\ref{sec:43}, and we also keep our convention on differentiation with respect to $u$, $\rho$, or $m$, as defined ealier.
First of all, we observe that, for any smooth and compactly supported function $\psi$, the entropy pair  $(\eta^\psi,q^\psi)$ satisfies the estimates 
\be
\label{boundpsi}
\aligned
& |\eta^\psi(\rho,u)| \lesssim \rho \langle \rho \rangle^{-\theta}, 
\qquad 
&& \left| \eta^\psi_m (\rho,u) \right| \lesssim \langle \rho \rangle^{-\theta},
\\
& \left| \eta^\psi_{mu}(\rho,u) \right| \lesssim \langle \rho \rangle^{-\theta}, 
&&\left| \eta^\psi_{m \rho}(\rho,u) \right| \lesssim \rho^{\theta-1} \langle \rho \rangle^{-\theta},
\qquad
&&& \left| q^\psi \right| \lesssim \rho,
\endaligned
\ee
as is easily checked from the expressions \eqref{eq:chidef}--\eqref{eq:etadef} of $\chi, \sigma$.
In order to establish the convergence property stated in Theorem~\ref{theo:12}, it now suffices to establish the following result. 

\begin{proposition}[Compactness of the entropy dissipation] 
\label{prop:compact}
Assume that the tame condition (TC) and the bounds~(\ref{eq:650}) hold.
Then, for all compactly supported test-functions $\psi$, the sequence $\partial_t \eta^\psi (\rho^\eps,u^\eps) + \partial_x q^\psi (\rho^\eps,u^\eps)$ is compact in $H^{-1}_{\operatorname{loc}}(\mathbb{R_+}\times\RR)$. 
\end{proposition}

\begin{proof} {\bf 1. Conservation law for $\eta^\psi$.} It reads
$$
\partial_t \eta^\psi (\rho,u) + \partial_x q^\psi (\rho,u) = \eta^\psi_m 
\left(
  (\mu^\eps(\rho) u_x)_x + \left(\rho \kappa^\eps(\rho)\rho_{xx} + \frac{1}{2} (\rho {\kappa^\eps}'(\rho) - \kappa^\eps(\rho))(\rho_x)^2\right|_x \right).
$$
We fix from now on a compact set $K$, a time $T>0$, and we will show the compactness of $\partial_t \eta^\psi (\rho^\eps,u^\eps) + \partial_x q^\psi (\rho^\eps,u^\eps)$ in $H^{-1}([0,T]\times K)$.

\vskip.15cm  

\noindent \textbf{2. Compactness of the viscous term in $W^{-1,q}_{\operatorname{loc}}$ for $q<2$} We deal first with the viscous term 
$\eta^\psi_m (\mu^\eps(\rho) u_x)_x$, which can be written
$$
\eta^\psi_m (\mu^\eps(\rho) u_x)_x = \underbrace{\left( \eta^\psi_m \mu^\eps(\rho) u_x \right|_x}_\NI
- \underbrace{\eta^\psi_{m \rho} \rho_x \mu^\eps(\rho) u_x}_\NII
- \underbrace{\eta^\psi_{m u} u_x \mu^\eps(\rho) u_x}_\NIII.
$$
Using successively the estimates~\eqref{boundpsi} and the bound on $\frac{\mu^\eps(\rho) \rho}{p(\rho)} \in L^\infty_{t,x}$, we obtain 
$$
\left| \mu^\eps(\rho) u_x \eta^\psi_m \right| \lesssim \mu^\eps(\rho) \left| u_x \right| \lesssim 
\sqrt{\mu^\eps(\rho)} |u_x| \sqrt{\mu^\eps(\rho)} = o(1) \sqrt{\mu^\eps(\rho)} |u_x|
$$
where $o(1)$ is understood in the $L^\infty_{t,x}$ topology and $\eps \to 0$; namely, $\mu^\eps(\rho) = 0(1)$ follows from 
$\mu^\eps(\rho) \lesssim \inf(\epsilon \mu(\rho),\langle \rho \rangle^{-\theta})$. Therefore, we find 
\be
\label{estimateI}
\NI\to 0 \qquad \mbox{in $H^{-1}([0,T] \times K))$ as $n \to +\infty$}.
\ee
Next, in view of the estimates~\ref{boundpsi} and by Cauchy-Schwarz inequality,  we have 
$$
|\NII| \lesssim \rho^{\theta-1} \mu^\eps(\rho) |u_x| |\rho_x| \lesssim \mu^\eps(\rho) |u_x|^2 
+  \rho^{2\theta-2} \mu^\eps(\rho) |\rho_x|^2 \lesssim  \mu^\eps(\rho) |u_x|^2 + \frac{\mu^\eps(\rho) p'(\rho)}{\rho^2} \rho_x^2.
$$
This implies a uniform bound for $\NII$ in the $L^1$ norm:
\be
\label{estimateII}
\left\| \NII \right\|_{L^1([0,T] \times K)} \lesssim 1.
\ee
Finally, by Lemma~\ref{boundpsi}, we have 
$
|\NIII| \lesssim \mu^\eps(\rho) u_x^2, 
$
thus
\be
\label{estimateIII}
\left\| \NIII \right\|_{L^1([0,T] \times K)} \lesssim 1.
\ee
Gathering~(\ref{estimateI}), (\ref{estimateII}), and (\ref{estimateIII}) gives that $\eta^\psi_m (\mu(\rho) u_x)_x$ is compact in $W^{-1,q}$ for all $q <2$. 

\vskip.15cm  

\noindent \textbf{3. Compactness of the capillary term in $W^{-1,q}$ for $q<2$} We now deal with the capillary term 
$ \eta^\psi_m (\rho \kappa^\eps(\rho)\rho_{xx} + \frac{1}{2}(\rho {\kappa^\eps}'(\rho) - \kappa^\eps(\rho))(\rho_x)^2)_x$ which can also be written as 
\be
\nonumber
\begin{split}
& \eta^\psi_m (\rho \kappa^\eps(\rho)\rho_{xx} + \frac{1}{2} (\rho {\kappa^\eps}'(\rho) - \kappa^\eps(\rho))(\rho_x)^2)_x =
\underbrace{\left( \eta^\psi_m \rho \kappa^\eps(\rho)\rho_{xx} \right|_x}_\NI 
- \underbrace{ \rho \kappa^\eps(\rho)\rho_{xx} \eta^\psi_{m \rho} \rho_x}_\NII
- \underbrace{\rho \kappa^\eps(\rho)\rho_{xx} \eta^\psi_{m u} u_x}_\NIII
 \\
& + \underbrace{\frac{1}{2} \left( \eta^\psi_m (\rho {\kappa^\eps}'(\rho) - \kappa^\eps(\rho))(\rho_x)^2 \right|_x}_\NIV
- \underbrace{\frac{1}{2} (\rho {\kappa^\eps}'(\rho) - \kappa^\eps(\rho)) (\rho_x)^2 \eta^\psi_{m \rho} \rho_x}_\NV
- \underbrace{\frac{1}{2} (\rho {\kappa^\eps}'(\rho) - \kappa^\eps(\rho)) (\rho_x)^2 \eta^\psi_{m u} u_x}_\NVI
\end{split}
\ee
Using the estimates~\eqref{boundpsi} and the condition (TC) gives
\be
\label{aigle1}
\left| \eta^\psi_m \rho \kappa^\eps(\rho)\rho_{xx} \right| \lesssim \rho \kappa^\eps \rho_{xx} 
\lesssim \sqrt{\mu^\eps(\rho)} \sqrt{ \frac{ \mu^\eps \kappa^\eps}{\rho} } \rho_{xx} = o(1) \sqrt{ \frac{\mu^\eps \kappa^\eps}{\rho} } \rho_{xx}.
\ee
Similarly, the estimates~\eqref{boundpsi}, the condition (TC), and Cauchy-Schwarz inequality yield 
\be
\label{aigle2}
\left| \NII \right| = \left| \rho \kappa^\eps(\rho)\rho_{xx} \eta^\psi_{m \rho} \rho_x \right| \lesssim  \rho \kappa^\eps(\rho) |\rho_{xx}| \rho^{\theta-1} |\rho_x| \lesssim 
\frac{\mu^\eps \kappa^\eps}{\rho} (\rho_{xx})^2 + \frac{\mu^\eps p'}{\rho^2} (\rho_x)^2.
\ee
Next, the estimates~\eqref{boundpsi}, the condition (TC), and Cauchy-Schwarz inequality give us 
\be
\label{aigle3}
\left| \NIII \right| 
= \left| \rho \kappa^\eps(\rho)\rho_{xx} \eta^\psi_{m u} u_x \right| \lesssim \left| \rho \kappa^\eps(\rho)\rho_{xx} u_x \right| \lesssim 
\frac{\mu^\eps  \kappa^\eps}{\rho} (\rho_{xx})^2 + \mu^\eps (u_x)^2,
\ee
To deal with $\NIV$, we use~\eqref{boundpsi} and the condition (TC) to obtain: 
\be
\label{aigle4}
\left| \eta^\psi_m (\rho {\kappa^\eps}'(\rho) - \kappa^\eps(\rho))(\rho_x)^2 \right| \lesssim \kappa^\eps(\rho) (\rho_x)^2
\lesssim \sqrt{\mu^\eps(\rho)} \sqrt{\frac{\mu^\eps \kappa^\eps}{\rho^3}} (\rho_x)^2 
= o(1) \sqrt{\frac{\mu^\eps \kappa^\eps}{\rho^3}} (\rho_x)^2.
\ee
Once again, \eqref{boundpsi}, the condition (TC), and Cauchy-Schwarz inequality give
\be
\label{aigle5}
|\NV| = \left|  (\rho {\kappa^\eps}'(\rho) - \kappa^\eps(\rho)) (\rho_x)^2 \eta^\psi_{m \rho} \rho_x \right| \lesssim | \kappa^\eps \rho^{\theta-1} (\rho_x)^2 \rho_x|
\lesssim \frac{\kappa^\eps \mu^\eps}{\rho^3} (\rho_x)^4 + \frac{\mu^\eps p'}{\rho^2} (\rho_x)^2,
\ee
Finally, the same arguments give
\be
\label{aigle6}
|\NVI| = \left| (\rho {\kappa^\eps}'(\rho) - \kappa^\eps(\rho)) (\rho_x)^2 \eta^\psi_{m u} u_x \right|
\lesssim |\kappa^\eps \rho_x^2 u_x|
\lesssim \mu^\eps (u_x)^2 + \frac{\kappa^\eps \mu^\eps}{\rho^3} (\rho_x)^4.
\ee
The estimates~(\ref{aigle1}) and~(\ref{aigle4}) imply, together with the estimates listed in \eqref{eq:650} , 
$$
\NI + \NIV \to 0 \qquad \mbox{in $H^{-1}(K\times [0,T])$) as $n \to +\infty$},
$$
whereas the estimates~(\ref{aigle2}), (\ref{aigle3}), (\ref{aigle5}), and (\ref{aigle6}) entail, together with the estimates listed in \eqref{eq:650} ,
$$
\left\| \NII + \NIII + \NV + \NVI \right\|_{L^1(K\times [0,T])} \lesssim 1.
$$
Combining the last two statements implies that the capillary term is compact in $W^{-1,q}$ for $q < 2$.

\vskip.15cm  

\noindent \textbf{4. Interpolation argument and conclusion.}
In view of~\eqref{boundpsi}, the functions $|\eta^\psi|$ and $|q^\psi|$ are bounded by $\rho$. Since $\rho$ is bounded in $L^{\gamma+1}$, it follows that $\partial_t \eta^\psi + \partial_x q^\psi$ is bounded in $W^{-1,\gamma+1}$, with of course $\gamma+1>2$. Interpolating this property with the compactness in $W^{-1,q}$ for all $q <2$ gives the desired result.
\end{proof}


\subsection{Completion of the proof of Theorem~\ref{theo:12}}

Following \cite{LW}, we can associate to the sequence $(\rho^\eps, u^\eps)$ a Young measure
 $\nu: \RR_+ \times \RR \to \text{Prob}(\RR_+ \times \RR)$, satisfying the weak convergence property \eqref{eq:nu1}
for all continuous functions $f=f(\rho,u)$ satisfying a certain growth condition. 
More precisely, the higher-integrability property of the density allows us to check that 
\eqref{eq:nu1} holds for all $| f | \leq f_0 (1 + \rho^{\gamma+1})$ with $\lim_{\rho \to + \infty} f_0(\rho) = 0$. 
This is sufficient to imply that the entropy $\eta^\psi(\rho^\eps, u^\eps)$ and 
the entropy flux $q^\psi(\rho^\eps, u^\eps)$ converge to $\la \nu, \eta^\psi \ra$ and $\la \nu, q^\psi \ra$ , respectively, 
{\sl provided} the function $\psi$ is {\sl compactly supported.}
 This convergence property suffices to state Proposition~\ref{prop:compact}, above. 

However, in order to recover the conservation laws in the Euler system and the global energy inequality, 
the additional integrability property for the velocity is required, which allows us now to use {\sl sub-cubic functions} $\psi$ for dealing with
 the entropies, and 
{\sl sub-quadratic functions} $\psi$ in the entropy flux. Observe that the local energy identity does not make sense at the 
level of regularity and integrability under consideration in the present paper. 

Finally, by applying the reduction lemma established in  \cite{LW} for Young measures satisfying Tartar's equation and associated 
with polytropic fluids, we conclude that $\nu$ is a Dirac mass or else is supported on the vacuum line. This completes the proof of 
 Theorem~\ref{theo:12}. We refer to \cite{LW} as well as \cite{GL3} for further details and generalizations, including a framework covering
 real fluids.
 

\section*{Acknowledgments}  

The first author (PG) was partially supported by NSF grant DMS-1101269, a start-up grant from the Courant Institute, and a Sloan fellowship.
This work was completed when the second author (PLF) was visiting the Courant Institute in 2011 and 2012 and was also supported by ANR grant SIMI-1-003-01. 


\end{document}